\documentclass[10pt]{amsart}
\usepackage{amsopn}
\usepackage{amsmath, amsthm, amssymb, amscd, amsfonts, xcolor}
\usepackage[latin1]{inputenc}

\usepackage{xcolor}

\usepackage[pdftex]{hyperref}
\theoremstyle{plain}

\newtheorem{teo}{Theorem}[section]
\newtheorem{cor}[teo]{Corollary}
\newtheorem{cosa}[teo]{Facts}
\newtheorem{prop}[teo]{Proposition}
\newtheorem{lema}[teo]{Lemma}

\theoremstyle{definition}

\theoremstyle{remark}
\newtheorem{nota}[teo]{Remark}

\numberwithin{equation}{section}

\makeatletter\renewcommand\theenumi{\@roman\c@enumi}\makeatother

\title{Normal holonomy of orbits and Veronese submanifolds}

\author{Carlos Olmos}

\author{Richar Fernando Ria\~ no-Ria\~no}

\address{Facultad de Matem\'atica, Astronom\'ia y F\'isica,
  Universidad Nacional de C\'ordoba, Ciudad Universitaria, 5000
  C\'ordoba, Argentina}
\email{\href{mailto:olmos@famaf.unc.edu.ar}{olmos@famaf.unc.edu.ar}
\ \ \ \
\href{mailto:riano@famaf.unc.edu.ar}{riano@famaf.unc.edu.ar}}

\thanks {{\bf Supported by:} FaMAF-Universidad Nacional de C\' ordoba and CIEM-Conicet.
}

\date {\today}
\begin{document}

\maketitle

\begin {abstract}
 It was conjectured, twenty years ago,
the following result that
would generalize the so-called rank rigidity theorem for
homogeneous Euclidean submanifolds: let $M^n$, $n\geq2$, be
a full and irreducible  homogeneous
submanifold of the sphere $S^{N-1}\subset \mathbb {R}^N$
 and such that the normal
holonomy group is not transitive (on the unit sphere of the
normal space to the sphere). Then $M^n$ must be an orbit of an
irreducible
$s$-representation (i.e. the isotropy representation of a semisimple
Riemannian symmetric space).

If $n=2$, then the normal holonomy is always transitive,
unless $M$ is a homogeneous isoparametric hypersurface
of the sphere (and so the conjecture is true in this case).
We prove the conjecture when $n=3$.
In this case $M^3$ must be either isoparametric or
a Veronese submanifold.
The proof combines geometric arguments with
(delicate) topological arguments that uses information from
two different fibrations  with the same total space
(the holonomy tube and the caustic fibrations).

We also prove the conjecture   for $n\geq 3$ when the normal holonomy
acts irreducibly and the codimension is the maximal possible
$ \frac 1 2 n(n+1)$. This gives a characterization of Veronese submanifolds
in terms of normal holonomy.
We also extend this last result by replacing the homogeneity assumption
by the assumption of minimality (in the sphere).

Another result of the paper, used for the case $n=3$, is that the number of
irreducible factors of the local normal holonomy group, for any
Euclidean submanifold $M^n$, is less or equal than  $[\frac n 2]$ (which
is the rank of  the orthogonal group $\text {SO}(n)$).
This bound is sharp
and improves the known bound $\frac 1 2 n(n-1)$.

\end {abstract}

\section {Introduction}

The holonomy of the normal connection
turns out to be a useful tool in Euclidean
submanifold geometry \cite {BCO}. The most important
applications
of this tool
were the alternative  proof of Thorbergsson theorem
\cite {Th}, given in
\cite {O2},
 and   the rank rigidity theorems for
 submanifolds \cite {O3, CO, DO}
(see Section 2.1). Moreover, the extension of Thorbergsson's result
to infinite dimensional geometry,  given by \cite {HL},
makes also use
of normal holonomy.

It is interesting to remark that
normal holonomy is related, in a very subtle way, to
Riemannian holonomy. Namely, by using submanifold geometry,
with normal holonomy ingredients, one can give short and geometric proofs
of both Berger holonomy theorem \cite {B} and Simons
holonomy (systems) theorem \cite {S} (see \cite {O4, O5}). Moreover,
by applying this
methods, it was proved in \cite {OR} the so-called skew-torsion holonomy
theorem with applications to naturally reductive spaces.

The starting point  for this theory was the
normal holonomy theorem
\cite {O1} which asserts that the (restricted) normal
holonomy group representation,
of
 a submanifold of a space form, is, up to a trivial factor, an $s$-representation.
(equivalently,  the normal holonomy is a Riemannian non-exceptional
holonomy). This implies that the so-called principal holonomy tubes have
flat normal bundle (holonomy tubes are the image, under the normal exponential map,
of the holonomy subbundles of the normal bundle). Such tubes, despite to the classical
spherical tubes, behaves nicely with respect to products of submanifolds.

But the normal holonomy,  which is invariant under conformal transformations of the
ambient space,
 gives much weaker information in
submanifold geometry than the Riemannian
holonomy in Riemannian geometry.
For instance, the reducibility of the normal holonomy representation
does not imply that the manifold splits. So, interesting applications of
the normal holonomy can be expected only within a
restrictive class of submanifolds. For instance:

 ({\it 1}) submanifolds with constant principal curvatures,

 ({\it 2}) complex submanifolds of the complex projective space

({\it 3}) homogeneous submanifolds.

\,

For the first two classes of submanifolds there are ``Berger-type" theorems.

For (1) one has  the following reformulation of the Thorbergsson
theorem \cite {Th}:
\it a full and irreducible  submanifold with constant principal curvatures,
such that the normal holonomy, as a submanifold of the sphere, is non-transitive
must be  either a inhomogeneous isoparametric hypersurface or an orbit of an
$s$-representation. \rm

For (2) we have the following result \cite {CDO}: \it a complete
full and irreducible
complex submanifold $M$ of the complex projective space with non-transitive
normal holonomy is  the complex orbit (in the projectivized tangent space) of the
isotropy representation of a Hermitian symmetric space or, equivalently,
$M$ is extrinsically symmetric \rm . This result is not true without the completeness
assumption.

For the class (3)  we have the rank rigidity theorem for submanifolds \cite {O3, DO}:
\it if the normal holonomy  of a full and irreducible Euclidean
homogeneous submanifold $M^n =K.v$, $n\geq 2$
 has a fixed non-null
vector, then $M$ is contained in a
sphere. If the dimension of the fixed set of the normal
holonomy has dimension at least
$2$, then $M$ is an orbit of an $s$-representation
(perhaps by enlarging the group $K$). \rm

But this last result would be only a particular case of a Berger-type result
that it was conjectured twenty years ago in \cite {O3}: \it if the normal holonomy
of a full and irreducible homogeneous submanifold $M^n$
 of the sphere,  $n\geq 2$,   is
non-transitive then $M$ is an orbit of an $s$-representation. \rm

For $n=2$ the normal holonomy must be always transitive or trivial
(see \cite {BCO}, Section 4.5 (c)).

The goal of this article is twofold. On the one hand, to give some progress on
this conjecture. On the other hand, to characterize the classical (Riemannian)
Veronese submanifolds in terms of normal holonomy.

If a submanifold $M^n$ of the sphere has irreducible
and non-transitive normal holonomy, then the
first normal space, as a Euclidean submanifold,
coincides  with the normal space
(see Remark \ref {correction}). This imposes the restriction
 that the codimension is at most
$\frac 1 2n(n+1)$. We will prove the above mentioned conjecture in the
case that the normal holonomy acts irreducibly  and
the (Euclidean) codimension is the maximal one
 $\frac 1 2n(n+1)$. The proof uses most of  the techniques  of the theory.
Moreover, the most difficult case is in dimension $n=3$ for which we have
to use also
 delicate topological arguments involving two different fibrations on
a partial holonomy tube: {\it the holonomy tube fibration} and {\it
the caustic fibration}.

\,

We extend these resuls by replacing the homogeneity by the property that
the submanifold is minimal in a sphere. But the proof of this result is simpler
than  the homogeneous case and a  general
proof works also for $n=3$.

\,

We also prove the sharp bound $\frac n 2$ on the number of irreducible
factors of the normal holonomy, which implies, from the above mentioned result,
the conjecture for $n=3$ (see Proposition \ref {cota}).

Let us explain our main results which are related to the so-called
Veronese submanifolds.

The isotropy representation of the
symmetric space $\text {Sl}(n+1)/
\text {SO}(n+1)$ is naturally  identified with the action of
$\text {SO}(n+1)$, by conjugation,  on the traceless
 symmetric matrices. A Veronese (Riemannian)
submanifold $M^n$, which has
parallel second fundamental form,  is
the orbit of a matrix with exactly two eigenvalues, one of which
has multiplicity $1$. Being $M$ a submanifold with constant principal curvatures,
the first normal space $\nu ^1 (M)$ coincides with the normal space $\nu (M)$.
Moreover, $\nu ^1(M)$ has maximal dimension. Namely, the codimension of $M$ is
$\frac {1}{2}n(n+1)$.

The restricted normal holonomy of $M$, as a submanifold
of the sphere, is the image, under the slice representation, of the
(connected) isotropy. Then the normal holonomy representation of $M$ is
 irreducible and  it is  equivalent to the isotropy representation of
 $\text {Sl}(n)/\text {SO}(n)$. So, the normal holonomy of $M$ is non-transitive
if and only if $n\geq 3$.
We have the following geometric
characterization of Veronese submanifolds in terms of
normal holonomy, which proves a special case of the conjecture on normal
holonomy of orbits, when the  normal holonomy, of a submanifold of the sphere,
 acts irreducibly, not transitively
 and
the codimension is maximal.

\,

{\bf Theorem A.} \it Let $M^n  \subset S ^ {n -1 + \frac 1 2 n(n+1)}$,
$n\geq 3$,
be a
homogeneous submanifold of the sphere.
Then $M$ is a (full) Veronese submanifold if and only if
 the restricted
normal holonomy group of $M$
acts irreducibly and not transitively.

\rm

\

For dimension $3$  the conjecture on normal holonomy is true.
Namely,

\,

{\bf Theorem B.} \it
Let $M^3  \subset S ^ {N-1}$
be a full  irreducible
homogeneous $3$-dimensional submanifold of the sphere.
 Assume  that the restricted
normal holonomy group
of $M$ is non-transitive. Then $M$ is an orbit of an $s$-representation.
Moreover,
 $M$ is either
 a principal orbit of the isotropy representation of
\, $\text {Sl}(3)/\text {SO}(3)$\,   or a Veronese  submanifold.
\rm

\

The irreducibility and fullness condition on $M$ is always with respect to
the Euclidean ambient space.

We can replace, in Theorem A,  the homogeneity condition
by the assumption of minimality in the sphere.

\,

{\bf Theorem C.} \it
Let $M^n$, $n\geq 3$, be a complete
(immersed)
submanifold of the sphere $S^{n-1 + \frac {1}{2}n(n+1)}$.
Then $M^n$ is, up to a cover, a (full) Veronese
submanifold if and only if
$M$ is a minimal submanifold
and
 the restricted
normal holonomy group
 acts irreducibly and not transitively.
\rm

\,

The assumptions of homogeneity or minimality, in our main results,
cannot be dropped, since a conformal (arbitrary) diffeomorphism of the
sphere transforms $M$ into a submanifold with the same
normal holonomy but in general  not any more minimal.
Last theorem admits a local version.

\,

We will explain the main ideas in the proof of Theorem A, when
$n\geq 4$.

Let $\tilde A$ be the traceless shape operator of $M = H.v$, i.e.
$\tilde A_\xi = A_\xi - \frac 1 n \langle H , \xi \rangle Id$,
where $H$ is the mean curvature vector. Let us consider the
 map $\tilde A$, from the normal space
$\bar {\nu} _q (M)$ to sphere into the traceless symmetric endomorphisms
$Sim _0 (T_qM)$. Then $\tilde A$ maps normal spaces to the $\Phi (q)$-orbits into
normal spaces to the $\text {SO}(n)$-orbits, by conjugation, in $Sim _0 (T_qM)$.
By using the results in Section 2, which are related to Simons theorem, we obtain
that $\tilde A$ is a homothecy which maps the normal holonomy group $\Phi (q)$
into $\text {SO}(n)$. This implies that the eigenvalues of $\tilde A_{\xi}$
do not change
if $\xi$ is parallel transported along a loop.
From the homogeneity, since the group $H$ is always inside
the $\nabla ^\perp$-transvections, we obtain that the eigenvalues of
$\tilde A _{\xi (t)}$ are constant, if $\xi (t)$ is a parallel normal field along a
curve. Now we pass to an appropriate,  singular, holonomy tube, $M_\xi$, where
$A_\xi$ has exactly two eigenvalues one of them of multiplicity $2$.
 Let
$\hat {\xi}$ be the parallel normal field of $M_\xi$ such that $M$ coincides with
the parallel focal manifold $(M_\xi)_{-\hat {\xi}}$ to $M_\xi$.
One obtains that the  three eigenvalue functions,
$\hat {\lambda} _1, \hat {\lambda} _2$ and $\hat {\lambda} _3 = -1$,
 of the shape operator $\hat A _{\hat {\xi}}$ of $M_\xi$ have constant multiplicities.
The two horizontal eigendistributions of
$\hat A _{\hat {\xi}}$, let us say
 $E_1$ and $E_2$, have multiplicities $2$ and $(n-2)$ respectively.
The vertical
distribution is the eigendistribution associates to the constant eigenvalue $-1$.
From the above mentioned properties of $\tilde A$ and the tube formulas one obtains
that $\hat{\lambda} _1$ and $\hat{\lambda} _2$ are functionally related (so if one eigenvalue
is constant along a curve the other is also constant).
From the Dupin condition, since $\text {dim} (E_1)\geq 2$,
 $\hat{\lambda} _1$, and so $\hat{\lambda} _2$, as previously
remarked, are constant along the integral manifolds of $E_1$.
If $n\geq4$, the the same is
true for the distribution $E_2$. So, the eigenvalues of
$\hat A _{\hat {\xi}}$
are constant along horizontal curves. But any two points in a holonomy tube
can be joined by a horizontal curve. Then  $\hat A _{\hat {\xi}}$ has constant
eigenvalues and so $\hat {\xi}$ is an isoparametric non-umbilical parallel
normal field.
Then,  by the
isoparametric rank rigidity theorem, the holonomy tube $M_\xi$, and
therefore $M$, is an orbit of an $s$-representation.
From this we prove, without using classification results, that $M$
must be a a Veronese submanifold.

If  $n=3$, the proof is much harder, since the Dupin condition
does not apply for $E_2$, and requires
topological arguments, not valid for $n >3$, as pointed out before.

\

\section {Preliminaries and basic facts}

In this section, as well as in the appendix,
for the reader convenience, we  recall the basic notions
and results that are needed in this article. We also include in this part
some new results that  are auxiliary for our purposes. Some of them
have a small interest  in its own right, or the proofs are different from
the standard ones.

The general reference for this section is \cite {PT, Te, BCO}.

\subsection {Orbits of $s$-representations and  Veronese submanifolds.}

\

A submanifold  $M\subset \mathbb {R}^N$ has \it constant principal
curvatures \rm if the shape operator $A_{\xi (t)}$ has constant eigenvalues,
for any $\nabla ^\perp$-parallel normal vector field $\xi (t)$ along any
arbitrary  (piece-wise differentiable) curve $c(t)$ in $M$.
If, in addition, the normal bundle of $\nu (M)$ is flat, then $M$ is called
isoparametric.

A submanifold $M$ with constant principal curvatures (extrinsically) splits as
$M = \mathbb {R}^k \times M'$, where $M'$ is  compact and
contained in a sphere.

The (extrinsic) homogeneous isoparametric submanifolds are exactly the principal
orbits of polar representations \cite {PT}. The other orbits have
constant principal curvatures (and, in particular, this family of orbits contains
the submanifolds with parallel second fundamental form). But it is
 not true that  all homogeneous submanifolds
with
constant principal curvatures are orbit of polar representations (there exists
a homogeneous
focal parallel manifold to an inhomogeneous isoparametric hypersurface of the sphere
\cite {FKM}). It turns out, from Dadok's classification \cite {Da},
that polar representations are orbit-like equivalent to the
so-called $s$-representations, i.e. the isotropy representations of semisimple
simply connected Riemannian symmetric spaces. So, a full and homogeneous
(not contained in a proper affine
subspace) Euclidean submanifold $M$
is isoparametric
if and only if it is a principal orbit of an $s$-representation.
It is interesting to remark that there is a
classification free proof \cite {EH},
for
cohomogeneity different from 2, of the fact that any  polar representation is
orbit-like to an $s$-representation.

One has the following remarkable result

\begin {teo} \rm (Thorbergsson, \cite {Th, O3}). \it
 A compact full irreducible  isoparametric Euclidean submanifold of
codimension at least $3$
is homogeneous {\rm (and so the orbit of an irreducible $s$-representation)}.
\end {teo}

The {\it rank}  at $p$, of a Euclidean submanifold $M$, $\text {rank}_p(M)$,
is
the maximal number of linearly independent parallel normal fields, locally defined
around $p$. The rank of $M$, $\text {rank}(M)$, is the minimum, over $p\in M$, of
$\text {rank}_p(M)$. If $M$ is homogeneous then $\text {rank}_p(M) =
\text {rank}(M)$, independent of $p\in M$.
The submanifold $M$ is said to be of {\it higher rank} if
its rank is at least $2$.

One has the following important result.

\begin {teo}\label {rank rigidity}
\rm (Rank Rigidity for Submanifolds, \cite {O3,O6, DO, BCO})
\it Let $M^n$, $n\geq 2$, be a Euclidean homogeneous
submanifold which is full and irreducible. Then,

\rm (a) \it If $\text {rank} (M)\geq 1$, if and only if
 $M$ is contained in a sphere.

\rm (b) \it If $\text {rank} (M)\geq 2$, then $M$ is an
orbit of an $s$-representation.

\end {teo}

A parallel normal field $\xi$ of $M$ is called {\it isoparametric} if
the shape operator $A_\xi$ has constant eigenvalues.
If the shape operator $A_\xi$, of a parallel isoparametric  normal field,
is umbilical, i.e.  a multiple $\lambda$ of the identity,
then $M$ is contained in a sphere,
if $\lambda \neq 0$, or $M$ is not full, if $\lambda = 0$.

One has the following result (see, \cite {BCO},
Theorem  5.5.2 and Corollary 5.5.3).

\begin {teo}\label {local isoparametric}
\rm (isoparametric local rank rigidity, \cite {CO}).
 \it Let $M^n$ be a full (local) and locally irreducible
submanifold
of $S^{N-1} \subset \mathbb {R}^N$ which admits a non-umbilical
parallel isoparametric normal field. Then $M$ is an inhomogeneous
isoparametric hypersurface or $M$ is (an open subset of)
an orbit of an $s$-representation.
\end {teo}

One has also a global version of the above result (see
\cite {DO} Theorem 1.2 and \cite {BCO}, Section 5.5 (b)).

\begin {teo}\label {global isoparametric} \rm
(isoparametric rank rigidity, \cite {DO}).
\it
  Let $M^n$  be a connected, simply connected and complete
Riemannian manifold and let
$f : M \to \mathbb {R^N}$ be an irreducible isometric immersion.
If there exists
a non-umbilical isoparametric parallel normal section
then $f : M \to \mathbb {R^N}$ has constant
principal curvatures {\rm (an so, if $f(M)$ is not an isoparametric
hypersurface of a sphere,
then it is an orbit of an $s$-representation)}.

\end {teo}

\

Let $K$ acts (by linear isometries) on $\mathbb {R}^{N}$ as an
$s$-representation.
Let
$(G,K)$ be  the associated  simple (simply connected) symmetric pair with
 Cartan decomposition $\mathfrak {g} = \mathfrak  {k} \oplus
\mathfrak {p}$,
where $ \mathfrak {p} \simeq \mathbb {R}^N$. Let
$M= K.v$ be an  orbit.

One has that the normal space to $M$ at $v$ is given by \cite {BCO}
$$\nu _v (M) = \mathcal {C} (v): =
\{ x \in \mathfrak {p} : [x,v] = 0\} \ \ \ \ \ \ \ \text {(*)} $$
where $[\, ,\, ]$ is the bracket of $\mathfrak {g}$.

An $s$-representation is always the product of irreducible ones.
Then the orbit $M = K.v$ is a full submanifold if and only if
all the  components of $v$, in any $K$-irreducible subspace of
$\mathbb {R}^N$,
are not zero.

Let  $M$ be a full orbit of an $s$-representation and let $p\in M$. Then the map
$\xi \mapsto A_{\xi}$, from $\nu _p (M)$ into the symmetric endomorphisms of
$T_pM$, is injective. In other words, the first normal space of $M$
at $p$ coincides with the normal space (see \cite {BCO}).

One has the following result from \cite {HO}; see also
\cite {BCO}, Theorem 4.1.7.

\begin {teo} \label {slice holonomy} {\rm (\cite {HO})}
Let $K$ acts on $\mathbb {R}^{N}$
as an  $s$-representation and let $M= K.v$ be a full orbit.
Then the   normal holonomy group $\Phi (v)$
of $M$ at $v$ coincides
with the image of the representation of the isotropy $K_v$ on
 $\nu _v (M)$ {\rm (the so-called} {\it slice representation}{\rm )}.
\end {teo}

For a Euclidean vector space
$(\mathbb {V}, \langle \, , \, \rangle )$, let $Sim  (\mathbb {V})$
denote the vector space of  (real) symmetric endomorphisms of
$\mathbb {V}$. The inner product on $Sim  (\mathbb {V})$ is the usual
one,  $\langle A ,B \rangle = \text {trace} (A.B)$.

We denote by $Sim _0 (\mathbb {V})$ the vector space of traceless
symmetric endomorphisms.

\begin {cor} \label {mainco} Let $K$ acts (by linear isometries)
on $\mathbb {R}^{N}$
as an  $s$-representation and let $M= K.v$,
where $\vert v \vert =1$.
Assume that the normal
holonomy group $\Phi (v)$ acts irreducibly on
$\bar {\nu }_v (M) : =
\{v\}^\perp \cap \nu _v (M)$. Then $M$ is a minimal submanifold
of the sphere $S^{N-1} \subset \mathbb {R}^{N}$. Moreover, the map
$\xi \mapsto A_{\xi}$ is a homothecy, from $\bar {\nu}_v (M)$ onto
its image in $Sim _0 (T_vM)$.
\end {cor}

\begin {proof}
The mean curvature vector $H(v)$ must be  fixed by the isotropy, represented
on the normal space.  Then, from Theorem \ref {slice holonomy},
$H(v)$ must be fixed by
$\Phi (v)$. Then, from the assumptions, $H(v)$
must be proportional to $v$ (which is
fixed by the normal holonomy group). Then $M$
is a minimal submanifold of the sphere.

Let us consider the following inner product $(\, ,\, )$ of $\bar {\nu }_v(M)$:
$(\xi , \eta) = \langle A_{\xi} , A_{\eta} \rangle$
Then, $(\,,\, )$ is $\Phi (v)$-invariant. In fact,
if $\phi \in \Phi (v)$,
there exists, from Theorem \ref {slice holonomy}, $g\in K_v$ such that
$g_{\vert \bar {\nu}_v (M)} = \phi$. Then
$$(\phi (\xi), \phi (\eta)) =
(g.\xi , g.\eta) = \langle A_{g.\xi} , A_{g.\eta} \rangle
= \langle gA_{\xi}g^{-1} , gA_{\eta}g^{-1} \rangle
= \langle A_{\xi} , A_{\eta} \rangle = (\xi , \eta)$$

Since $\Phi (p)$ acts irreducibly, then $(\, ,\, )$ is proportional to
$\langle \, ,\, \rangle$. Then $\xi \mapsto A_{\xi} $ is a homothecy.

\end {proof}

Recall that the normal holonomy (group) representation,
of a submanifold of a space form, on the normal space,
is, up to the fixed set, an $s$-representation \cite {O1, BCO}.

\

The proof of the above mentioned result depends on the construction
of the so-called {\it adapted normal curvature tensor}
$\mathcal {R}^\perp$ (see \cite {O1} and \cite {BCO}, Section 4.3 c).
In fact, if $M$ is an arbitrary submanifold
of a space of constant curvature then $\mathcal {R}^\perp$ is an algebraic
curvature tensor on the normal space $\nu (M)$.
Namely, if $p \in M$ and $R^\perp $ is the normal curvature tensor
at $p$, regarded as a linear map for $\Lambda ^2 (T_pM) \to
\Lambda ^2 (\nu _p (M))$, the adapted
normal curvature tensor is defined by
$$ \mathcal {R}^\perp  = R^\perp  \circ (R^\perp )^t $$
where $(\ )^t$ is the transpose endomorphism.
This implies that $\mathcal {R}^\perp$ has the same image as
$R^\perp$.

From the Ricci identity one
has the nice formula, if
$\xi _1 ,  \xi _ 2 , \xi _ 3 , \xi _ 4 \in \nu _p (M)$,
$$ \langle \mathcal {R}^\perp _{\xi _1 ,\xi _2}\xi _3 , \xi _4 \rangle
= -\text {trace} ([A_{\xi _1}, A_{\xi _2}]\circ [A_{\xi _3}, A_{\xi _4}])
\ \ \ \ \ \ \ \ $$
 $$
= \langle [A_{\xi _1}, A_{\xi _2}], [A_{\xi _3}, A_{\xi _4}]\rangle
= -\langle [[A_{\xi _1}, A_{\xi _2}], A_{\xi _3}], A_{\xi _4}\rangle
\ \   \ \ \ \text {(**)}  $$
where $A$ is the shape operator of $M$.

Since  $\mathcal {R}^\perp (\Lambda ^2 (\nu _p (M))) =
\ {R}^\perp (\Lambda ^2 (T _p M))$,
one has that  $\mathcal {R}^\perp _{\xi _1, \xi _2}$
belongs to the normal holonomy
algebra at $p$ (since curvature tensors, take values in the holonomy
algebra).

\

Since the isotropy representation of a semisimple symmetric space coincides
with that of
the dual symmetric space, we may always assume that the symmetric space is
compact.
Let then $(G,K)$ be a compact simply connected symmetric pair and let
$\mathfrak {g} = \mathfrak {k} \oplus \mathfrak p$ be the
Cartan decomposition
associated to such a pair. The isotropy representation of $K$  is
naturally identified with the $\text {Ad}$-representation of $K$ on
$\mathfrak {p}$. The Euclidean metric on $\mathfrak {p}$ is $-B$, where
 $B$ is the Killing form of $\mathfrak {g}$. We denote by a dot
the $\text {Ad}$-action of $K$ on $\mathfrak {p}$. Let $0\neq v\in \mathfrak {p}$
 and let us consider the  orbit $M= K.v \simeq K/K_v$ which is a Euclidean
submanifold with
constant principal curvatures (and rank at least $2$ if and only if
it is not most singular).

Let us consider the restriction $\langle \, , \, \rangle$
of $-B$ to $\mathfrak {k}$. This is an $\text {Ad}$-K invariant
positive definite inner product on $\mathfrak {k}$.
Let us consider the (normally) reductive decomposition
$$\mathfrak {k} = \mathfrak {k}_v \oplus \mathfrak {m}$$
where $\mathfrak {k}_v$ is the Lie algebra of the isotropy
group $K_v$ and $\mathfrak {m}$ is the orthogonal complement,
with respect to $\langle \, , \, \rangle$, of $ \mathfrak {k} $.
The restriction of $\langle \, , \, \rangle$ to
$\mathfrak {m} \simeq T_{[e]}K/K_v \simeq T_{v}M$ induced
a so-called normal homogeneous metric on $M$, which is in particular naturally
reductive, that we also denote
by $\langle \, , \, \rangle$. Such a Riemannian  metric on $M$ will be called
 the canonical normal homogeneous metric.
In general this metric is different from
the induced metric as a Euclidean submanifold. Namely,

\begin {prop} Let $K$ acts on $\mathbb {R}^N$ as an
irreducible  $s$-representation
and let $M = K.v$, $v\neq 0$. If the (canonical)
normal homogeneous metric on $M$ coincides
with the induced metric, then $M$ has parallel second fundamental form
{\rm (or equivalently, $M$ is extrinsically symmetric \cite {Fe})}.

\end {prop}

\begin {proof}

We keep the notation previous to this proposition.
Let $\nabla ^c$ be the canonical connection on $M$
associated to the
reductive decomposition
$\mathfrak {k} = \mathfrak {k}_v \oplus \mathfrak {m}$.
Then the second fundamental form $\alpha$ of $M$ is parallel
with respect to the connection $\bar  {\nabla}^c =
\nabla ^c \oplus \nabla ^\perp$, i.e.
 $\bar  {\nabla}^c \alpha = 0$ \cite {OSa, BCO}.
Let $\bar {\nabla } = \nabla \oplus
\nabla ^\perp$, where $\nabla$ is the Levi-Civita
connection on $M$ associated to the induced metric
which coincides, by assumption,
with the normal homogeneous metric.
 Then
$$(\bar {\nabla} _x \alpha ) (y,z) = \alpha (D_x y, z)
+ \alpha (y , D_xz)$$
where $D = \nabla - \nabla ^c $
We have that $D_xy = -D_yx$. This is a general fact, for
naturally reductive spaces, since
 the canonical geodesics coincide with the Riemannian
geodesics (see, for instance, \cite {OR}).

Then
$$(\bar {\nabla} _x \alpha )(x,x) = 2\alpha (D_x x, x)
= 0$$

But, from the Codazzi identity,
 $(\bar {\nabla }_x\alpha )(y,z)$ is symmetric in all of its
three variables.
Then $\bar {\nabla }\alpha  = 0$ and so $M$ has parallel second
fundamental form.

\end {proof}

\begin {cor} \label {parallel} Let $K$ acts on $\mathbb {R}^N$
as an $s$-representation
and let $M = K.v$, $v\neq 0$. Assume that  $K_v$ acts irreducibly on $T_vM$.
Then $M$ has parallel second fundamental form {\rm
(or, equivalently, $M$ is extrinsically symmetric \cite {Fe})}.
\end {cor}

\begin {nota} \label {parallel-s} A submanifold of the Euclidean space
with parallel second fundamental form
is, up to a Euclidean factor, an orbit of an
$s$-representation \cite {Fe}
(see also \cite {BCO}).
\end {nota}

\begin {lema} \label {proportional}

Let $M^n, \bar {M} ^n  \subset S^{N-1}$ be submanifolds
of the sphere
with parallel second fundamental forms
(or, equivalently, extrinsically symmetric spaces). Assume also that
$M$ is a full submanifold of the Euclidean space
$\mathbb {R}^N$ and  that there exists
$p\in M \cap \bar M$ with $T_pM = T_p\bar M$.
Assume, furthermore,  that the associated fundamental forms at $p$,
$\alpha , \bar {\alpha}$
 of $M$ and $\bar M$, respectively, as submanifolds of the sphere,
are proportional (i.e. $\bar {\alpha} = \lambda  \alpha$, $\lambda \neq 0$).
Then $M = \bar M$  (and so $\lambda = 1$) or $M = \sigma (\bar {M})$, where
$\sigma$ is the orthogonal transformation of  $\mathbb {R}^N$ which
 is the identity on $\mathbb {R}p\oplus T_pM$ and minus the identity
on $\bar {\nu }_p(\bar {M}) = (\mathbb {R}p\oplus T_pM)^\perp$
(and so $\lambda
= -1$).

\end {lema}

\begin {proof}
Observe, in our assumptions,
 that the second fundamenal forms of $M$ and
$\bar M$, as Euclidean submaniofolds, are not proportional,
unless they coincide (since the shapes operators  of $M$ and
$\bar M$, coincides in the direction of the position vector $p$).

Let us write  $M= K.p$  where $K$ acts as an
irreducible $s$-representation. One
has that the restricted holonomy at $p$, of the bundle
$TM \oplus \bar {\nu} (M)$, is  the representation,
 of the connected isotropy
$(K_p)_0$, on
$T_pM \oplus \bar {\nu} _p(M)$.
This is a well-known fact that follows form
the following property: if $X$ belongs to the Cartan subalgebra
associated to the  symmetric pair $(K,K_p)$,
 then $\text {d}l_{\text {Exp}(tX)}$ gives
the Levi-Civita parallel transport, when restricted to $T_pM$,
along the geodesic
$\gamma (t) = \text {Exp}(tX).p$, and at the same time, when restricted to
$\bar {\nu} _p (M)$, the normal parallel transport along $\gamma (t)$.

Since curvature endomorphisms take values in the holonomy algebra, one
 has that $(R_{x,y}, R^\perp _{x,y})\in \mathfrak {t}_{p}$, where
$\mathfrak {t}_{p} = \text {Lie} (K_p) = \text {Lie} ((K_p)_0)\subset
\mathfrak {so}(T_pM)\oplus \mathfrak {so} (\bar {\nu} _p(M))  $ and $R$,
$R^\perp$ are the tangent and normal
curvature tensors of $M$ at $p$, respectively.

Let $R^S$ be the curvature tensor of the sphere $S^{N-1}$ at $p$, restricted
to $T_pM$.
Then, from the Gauss equation,
$$R_{x,y}  = T_{x,y} + R^S_{x,y}$$
where $\langle T_{x,y}z,w\rangle = \langle \alpha (x,w), \alpha (y,z)\rangle
- \langle \alpha (x,z), \alpha (y,w)\rangle $

For $\bar {M} = \bar {K}.p$ we have similar objects
$\bar {R},\,  \bar {R}^\perp ,\,  \bar {\mathfrak {t}}_{p}, \,  \bar {T}$.
From the assumptions one has that $\bar {T} = \lambda ^2  T$.
So,
$$\bar {R}_{x,y}  =
\lambda ^{2} T_{x,y} + R^S_{x,y}  \ \ \ \text {(a)}$$

From the assumptions, and Ricci equation, one has that
$$\bar {R}_{x,y}^\perp = \lambda ^{2} {R}_{x,y}^\perp
\ \ \ \ \ \ \ \ \ \ \ \ \text {(b)}$$

Now observe that, for any $X\in \mathfrak {t}_{p} \subset
\mathfrak {so}(T_pM)\oplus \mathfrak {so}(\bar{\nu} _p(M))$,
$$X.\alpha = 0 = X.(\lambda.\alpha) = X.\bar {\alpha}
\ \ \ \ \ \text {(c)} \ \ \ \  \ \ \ \ $$
and the same is true for any $\bar {X} \in \bar {\mathfrak {t}}_{p}$
(the actions of $X$ and $ \bar {X}$ are  derivations).

As we observed,
$(R_{x,y}, R^\perp _{x,y})\in \mathfrak {t}_{p}$,
and  $(\bar {R}_{x,y}, \bar {R}^\perp _{x,y})
\in \bar {\mathfrak {t}}_{p}$. Then, form
(a), (b) and  (c) one obtains, if $\lambda \neq \pm 1$
 that
$$ (R_{x,y}^S , 0). \alpha = 0 = (R_{x,y}^S , 0).
\bar {\alpha}$$.

Since the linear span of $\{R_{x,y}^S : x,y\in T_pM\}$
is $\mathfrak {so}(T_pM)$, one has that

$$\alpha (g.x, g.y) = \alpha (x,y)$$
for all $g\in \text {SO}(T_pM)$. Then,
from the  Gauss equation
$\langle A_\xi x, y\rangle = \langle \alpha (x,y), \xi\rangle$,
one obtains that all the shape operators of $M$ at $p$ commute
with any element of $\text {SO}(T_pM)$. Then $M$ is umbilical at $p$ and hence,
since it is homogeneous, at any point.
Then $M$ is an extrinsic  sphere. Since $M$ is full we conclude that
$M = S^{N-1}$. Then, since $n = N-1$,
$M =\bar {M}$.

Observe that the fullness condition is essential. In fact, if $M$ and $\bar
{M}$
are umbilical submanifolds of the sphere of different radios,
the second fundamental
forms at $p$ are proportional.

If $\lambda =1$, then $M$ and $\bar {M}$ have both the same second fundamental
form at $p$. Since  both submanifolds have parallel second
fundamental forms,
it is well-known and standard to prove that $M= \bar {M}$.

If $\lambda = -1$, then we replace $\bar {M}$ by $\sigma (\bar {M})$
and the second fundamental forms of $M$ and $\bar {M}$ must coincide.
Therefore, $M = \sigma (\bar {M})$.

\end {proof}

\begin {nota}\label {correction}

Let us enounce Theorem 4.1 in \cite {O5}:
\it
let $M^n$  be a locally full submanifold either of the Euclidean space or the
sphere, such that the local normal holonomy
group at $p$ acts without fixed non-
zero vectors.
Assume, furthermore, that no factor of the
normal holonomy is transitive on the
sphere. Then there are points in $M$,
arbitrary close to $p$, where the first normal
space coincides with the normal space.
In particular, $\text {codim}(M)\leq \frac 1 2 n(n + 1)$.
\rm

This bound on the codimension is correct.
But the better and
sharp estimate is $\text {codim}(M)\leq \frac 1 2 n(n + 1) -1$.
In fact, from the proof one has that if the shape operator, at a generic
$q\in M$,
$A_{\xi}$ is a multiple
of the identity (it needs not to be zero, as in that proof), then
$\xi$ is in  the nullity of the adapted normal curvature tensor
$\mathcal {R}^\perp$. But this last tensor is not degenerate.
This implies that the injective map
$A:\nu _q(M) \to Sim (T_qM)$ cannot be onto. Then $\text {dim} (\nu _q (M))
= \text {codim}(M)\leq \text {dim}(Sim (T_qM)) -1 = \frac 1 2 n(n + 1) -1$.

If $M$, in the above assumptions,
is is a submanifold of the sphere,
then the codimension of $M$,
as a Euclidean submanifold,
is bounded by $\frac 1 2 n(n + 1)$.

\end {nota}

\subsection {Holonomy systems}

\

We recall here some facts about holonomy systems that are useful in
submanifold geometry.

A {\it holonomy system} is a triple $[\mathbb {V}, R, H]$, where
$\mathbb {V}$ is a Euclidean vector space, $H$ is a connected
compact Lie subgroup of $\text {SO}(\mathbb {V})$ and   $R \neq 0$
is an algebraic Riemannian
curvature
tensor on $\mathbb {V}$ that takes values $R_{x,y}\in \mathfrak {h} =
\text {Lie}(H)$.
The holonomy system is called:

- {\it irreducible}, if $H$ acts irreducible on $\mathbb {V}$.

- {\it transitive}, if $H$ acts transitively on the unit sphere of $\mathbb {V}$.

- {\it symmetric}, if $h(R) = R$, for all $h \in H$.

\

Observe that a Lie subgroup  $H \subset \text {SO} (\mathbb {V})$
that acts irreducibly on
$\mathbb {V}$ must be compact, as it is well-known
(since the center of $H$ must be one-dimensional).

A holonomy system $[\mathbb {V}, R, H]$
is the product (eventually, after enlarging $H$) of irreducible holonomy systems
(up to a Euclidean factor).

One has the following remarkable result.

\begin {teo} \label {Simons} {\rm (Simons holonomy theorem, \cite {S, O5})}.

An irreducible and non-transitive holonomy system
$[\mathbb {V}, R, H]$ is symmetric.
Moreover, $R$ is, up to a scalar multiple, unique.

\end {teo}

\begin {nota}\label {+simons}
If $[\mathbb {V}, R, H]$ is an irreducible symmetric holonomy system,
then $\mathfrak {h}$ coincides with the  linear span of $R_{x,y}$, $x,y \in
\mathbb {V}$. In this case, since
$\langle R_{x,y}v,\xi\rangle = \langle R_{v,\xi}x,y\rangle$,   one has that
the normal space at $v$ to the orbit $H.v$ is given by
$$\nu _v(H.v) = \{\xi \in \mathbb {V}: R_{v,\xi}=0\}$$
\end {nota}

\,

From a symmetric holonomy system one can build an involutive algebraic Riemannian
symmetric
pair $\mathfrak {g} = \mathfrak {h} \oplus \mathbb {V}$.
The bracket  $[\, , \, ]$ is given by:

{\bf a)} $[\, , \, ]_{\vert \mathfrak {h}\times \mathfrak {h}}$
coincides with the bracket of $\mathfrak {h}$.

{\bf b)} $[X,v] = - [v,X] = X.v$, if $X\in \mathfrak {h}
\subset \mathfrak {so}(\mathbb {V})$ and
$v\in \mathbb {V}$.

{\bf c)} $[v,w] = R_{v,w}$, if $v,w \in  \mathbb {V}$.

\

This implies the following: \it if $[\mathbb {V}, R, H]$
is an irreducible and symmetric
holonomy system then $H$ acts on $\mathbb {V}$ as an
irreducible  $s$-representation. \rm

Observe that, in this case, the scalar curvature $sc(R)$ of $R$ is
different from $0$ (since this is true for the
curvature tensor of an irreducible
symmetric space).

\rm

\begin {lema} \label {T(R)} Let $[\mathbb {V}, R, K]$ be an
 irreducible and non-transitive holonomy system.
Let $T\in \text {SO}(\mathbb {V})$ be such
that
$R_{x,y}= 0$ if and only if $R_{T(x),T(y)}= 0$. Then
$T(R) = R$.
\end {lema}

\begin {proof}

Let $R'= T(R)$. If $\xi \in \nu _v(K.v) =
\{\xi \in \mathbb {V}: R_{v, \xi} = 0\}$,
then, from the assumptions, $R'_{v,\xi} =
T .R_{T(v),T(\xi )}.T^{-1} = 0$.
So, $ 0 = \langle R'_{v , \xi} x , y\rangle =
\langle R'_{x,y} v , \xi \rangle$,
for all $x,y \in \mathbb {V}$. Then the Killing field $R'_{x,y} \in
\mathfrak {so}(\mathbb {V})$ of
$\mathbb {V}$ is tangent to any orbit $K.v$. This implies that
$R'_{x,y} \in \tilde {\mathfrak {h}} = \text {Lie} (\tilde {K})$,
 where $\tilde {K} = \{g \in \text {SO}(\mathbb {V}): g\text { preserves any }
K\text {-orbit}\}$.
Observe that $\tilde {K}$ is a (compact) Lie subgroup of
$\text {SO}(\mathbb {V})$ which is non-transitive (on the unit sphere of
$\mathbb {V}$). Since $H\subset \tilde {K}$ we have that
$[\mathbb {V}, R, \tilde {K}]$ is also an
irreducible and non-transitive holonomy
system. From the Simons holonomy theorem we have that
$[\mathbb {V}, R, K]$ and $[\mathbb {V}, R,
\tilde {K}]$ are
both symmetric. Then $\mathfrak {h}$ and $\tilde {\mathfrak {h}}$
are (linearly) spanned by $R_{x,y}$, $x,y \in \mathbb {V}$.
Then $\mathfrak {h} = \tilde {\mathfrak {h}}$ and therefore,
$K = \tilde {K}$.

Since $R'$ takes values in $\tilde {\mathfrak {h}} = \mathfrak {h}$,
then $ [\mathbb {V}, R', K]$ is also an irreducible and non-transitive holonomy
system. Then, from the uniqueness part of Simons theorem,
$R' = \lambda R$, for some
scalar $\lambda \neq 0$. Since $T$ is an isometry, it induces an isometry on
the space of tensors. Then $\lambda = \pm 1$.
But $0\neq sc(R) = sc (R')$. Then $\lambda =1$ and hence $R' = R$.

\end {proof}

\begin {nota} \label {remarkholonomysystems} Let $M^n = K.v$, where
 $K$ acts (by linear isometries)
  on $\mathbb {R}^{n + \frac {1} {2} n(n+1)}$
as an $s$-representation ($\vert v\vert =1$).
Assume that the restricted
normal holonomy group $\Phi (v)$
acts irreducibly on $\bar {\nu}_v(M) = \{v\}^\perp
\cap \nu _v (M)$. In this case $M$ is a minimal submanifold of
the sphere $S^{n -1 + \frac {1} {2} n(n+1)}$ (see Corollary \ref {mainco}).

Let $A$ be the shape
operator of $M$ and let $Sim _0 (T_pM)$
be the space of  traceless
symmetric endomorphisms of $T_pM$.
Then the  map
$A: \bar {\nu}_v(M) \mapsto Sim _0 (T_vM) $
is a linear isomorphism. In fact, it is injective,
since the first normal space of $M$ coincides with the normal space, and
$\text {dim} (\bar {\nu} _v (M)) = \text {dim}(Sim _0(T_vM))$.
Moreover, by the second part of
Corollary \ref {mainco}, $A$ is a homothecy from
$\bar {\nu}_v(M)$  onto $ Sim _0 (T_vM)$, let us say, of constant $\beta > 0$.

Let us consider the following
two irreducible and symmetric holonomy systems:
$$[Sim _0 (T_pM) , R, \text {SO}(T_pM)] \text { \ and \ }
[\bar {\nu }_v (M), \mathcal {R}^\perp , \Phi (v)],$$
where $\mathcal {R}^\perp$ is the adapted normal curvature tensor of $M$
at $v$ and $R$ is the curvature tensor of $\text {Sl}(n)/\text {SO}(n)$
(which is explicitly  given by (***) of Section 1.3).

Observe that $[\bar {\nu }_v (M), \mathcal {R}^\perp , \Phi (v)]$
is symmetric since, by  Theorem \ref {slice holonomy},
the restricted normal holonomy group is given by
$$\Phi (v) = \{k_{\vert \nu _v(M)}: k \in (K_v)_0\}$$
and  $\mathcal {R}^\perp$ is left fixed by $K_v$.

Both algebraic curvature tensors are
related by the formula (**) of Section 1.1.
This implies that the homothecy $A$ maps
$\mathcal {R}^\perp$ into  $R$. Then the isometry $\beta ^{-1}A$ maps
$\mathcal {R}^\perp$ into $\beta ^4 R$.

Since in a symmetric irreducible holonomy system the Lie algebra of the
group is (linearly) generated by the curvature endomorphisms,
we conclude that $A$ maps $\Phi (v)$ onto $\text {SO}(T_pM)\simeq
\text {SO}(n)$. In particular, the two holonomy systems are equivalent and
$\Phi (v)\simeq \text {SO}(n)$.

\end {nota}

\subsection {Veronese submanifolds}

\

Let us consider the isotropy representation of the symmetric space
of the non-compact type
$X= \text {Sl}(n +1 )/\text {SO}(n+1)$
(which coincides with the isotropy representation of its compact dual
$\text {SU}(n+1)/ \text {SO}(n+1)$). The Cartan decomposition of such a space
is
$$\mathfrak {sl}(n+1) = \mathfrak {so}(n+1)\oplus  {Sim}_0 (n+1)$$
where ${Sim}_0 (n+1)$ denotes the traceless symmetric
(real) $(n+1)\times (n+1)$-matrices. The $\text {Ad}$-representation of
$\text {SO}(n+1)$ on $Sim _0 (n+1)$ coincides with
the action, by conjugation, of $\text {SO}(n+1)$ on  $Sim _ 0 (n+1)$.

The curvature tensor of $X$ at $[e]$ is given (up to a positive multiple)
by
$$R _{A,B}C = -[[A,B],C]$$
 and
$$ \langle R _{A,B}C , D\rangle  = -\langle [[A,B],C], D\rangle
= \langle [A,B], [C, D]\rangle \ \ \ \ \ \ \text {(***)}$$
where $A,B,C,D \in Sim _0 (n+1) \simeq T_{[e]}X$.

Let $S \in Sim _0 (n + 1)$ with exactly two eigenvalues, one of multiplicity
$1$ (whose  associated eigenspace we denote by  $E_1$)
and the other of multiplicity $n$ (whose
associated eigenspace we denote by  $E_2$).

The orbit
$V ^{n} =  \text {SO}(n+1) . S = \{k S k^{-1}: k \in \text {SO}(n+1)\} $
is called a {\it Veronese-type}
orbit (see Appendix).

The following assertions are easy to verify or  well-known.

\begin {cosa} \label   {Veronese facts}

\

\,

\rm

{ \ \it  (i)} The Veronese-type orbit \, $V^{n} = \text {SO}(n +1) . S$
\, is a full and irreducible submanifold of $Sim _0 (n+1)$ which
has dimension $n$ and codimension
$\frac 1 2 n (n+1)$. Moreover, $V^{n}$ is a minimal submanifold of the
sphere of radius $\Vert S\Vert$.

{\ \it  (ii)} An orbit of $\text {SO}(n+1)$ in $Sim _0 (n+1)$ has minimal dimension
if and only if it is of Veronese-type; see Lemma
\ref {minima}.

{\it (iii)} The  normal holonomy group at $S$, of the Veronese-type
orbit $V^{n}$,
coincides with image of the  slice representation of the
isotropy group
$(\text {SO}(n+1))_S = \text {S}(\text {O}(E_1)\times \text {O}(E_2))
\simeq
\text {S}(\text {O}(1)\times \text {O}(n))$. So, from (*), the restricted normal
holonomy representation, on $\bar {\nu }_S (V^{n}) =
\{S\}^\perp\cap {\nu }_S (V^{n})$,
is equivalent to the isotropy representation of the
symmetric space   $\text {Sl}(n)/ \text {SO}(n)$ of rank $n-1$. Then, this
normal holonomy representation is irreducible. Moreover, it is non-transitive
(on the unit sphere of $\bar {\nu }_S (V^{n})$) if and only if $n\geq 3$.

{\ \it  (iv)} A Veronese-type orbit
$V^{n} = \text {SO}(n+1).S = \text {SO}(n+1)/(\text {SO}(n+1))_S$
is intrinsically
a real projective space $\mathbb {R}P^{n}$. Moreover,
$(\text {SO}(n+1),(\text {SO}(n+1))_S)$ is a symmetric pair and so
$(\text {SO}(n+1))_S$ acts irreducibly on $T_SV^{n}$. Then, from
Corollary \ref {parallel}, $V^{n}$ has parallel second fundamental form
(as it is well known).

\qed

\end {cosa}

 A submanifold $M\subset \mathbb {R}^N$ is called a
{\it Veronese submanifold} if it is extrinsically isometric to a Veronese-type
orbit.

\begin {prop} \label {n-Veronese} Let $M^{n}=
K.v \subset \mathbb {R}^{n + \frac {1}{2} n(n+1)}$,
where $K$  acts on
$\mathbb {R}^{n + \frac {1}{2} n(n+1)}$
as an $s$-representation ($n\geq 2$).
Assume that  the restricted normal holonomy group $\Phi (v)$
 of $M$ at $v$,
restricted to
$\bar {\nu} _v (M) =  \{v\}^\perp \cap \nu _v (M)$,
acts  irreducibly (eventually, in a transitive way). Then,

\rm \ (i)\it \  The normal holonomy representation of $\Phi (v)$ on
$\bar {\nu} _v (M)$ is equivalent to the isotropy representation
of the symmetric space $\text {\rm Sl}(n)/\text {\rm SO}(n)$.

 \rm (ii) \it \ $M^n$ is a Veronese submanifold.
\end {prop}

\begin {proof}

Part {\it (i)} is a consequence of Remark
\ref {remarkholonomysystems}.

Since $K$ acts as an $s$-representation, then
  the image under the slice representation,
of the (connected) isotropy
group $(K_v)_0$,
coincides with the restricted  normal holonomy group $\Phi (v)$. But, from
part (i),
$\text {dim}(\Phi (v)) = \text {dim}(\text {SO}(n))$
Then the isotropy group $K_v$ has dimension at least $\text {dim} ({SO}(n))=
\text {dim} ({SO}(T_vM))$.

Observe that the isotropy representation of $K_v$ on $T_vM$
is faithful. Otherwise, $M$ would be
contained in the proper subspace which consists of the
 fixed vector of $K_v$ in  $\mathbb {R}^N$.

Then, $(K_v)_0 = \text {SO}(T_vM)$. So, $K_v$ acts irreducibly  on $T_vM$.
Then, from Corollary \ref {parallel}, $M$ has parallel second fundamental form.

Let $V^n$ be a Veronese submanifold of
$\mathbb {R}^{n+ \frac {1}{2}n(n+1)}$. We may
assume that $v\in V^n$ and that $T_vM = T_vV^n =
\mathbb {R}^n \subset \mathbb {R}^{n + \frac {1}{2}n(n+1)}$.
For $V^n$ we have, from Corollary \ref {mainco}
and Remark \ref {remarkholonomysystems},
 that its shape operator
 $\bar {A} : \{v\}^\perp \cap \nu _v (V^n) =
\{v\}^\perp \cap \nu _v (M) \to
Sim _0 (T_vV^n) = Sim _0 (T_vM)$  is a homothecy which induces an isomorphism
from the normal holonomy group
$\bar {\Phi} (v)$ of $V^n$ onto $\text {SO} (n) $.

The same is true, again from Corollary \ref {mainco}
and Remark \ref {remarkholonomysystems},
for the shape operator $A$ of $M$. Namely,
$ {A} : \{v\}^\perp \cap \nu _v (M)  \to
Sim _0 (T_vV^n) = Sim _0 (T_vM) = Sim _0 (\mathbb {R}^n)$
is a homothecy which induces an isomorphism
from the (restricted) normal holonomy group
$ {\Phi} (v)$ of $M$ onto $\text {SO} (n)$.
Then the map $ A^{-1} \circ \bar {A}$ is a homothecy
with constant, let us say,  $\beta > 0$, of the space
$\{v\}^\perp \cap \nu _p (M)$.
Let
 $h =
\beta ^{-1} A ^{-1} \circ \bar {A}$.
Then $h$ is a linear isometry  of
$\{v\}^\perp \cap \nu _v(M)$.

Let now $g$ be the linear isometry of $\mathbb {R}^ {n + \frac 1 2 n(n+1)}$ defined
by the following properties:

{\it \ \ (i)} $g(v)=v$.

{\it \  (ii)} $g_{\vert \{v\}^\perp \cap \nu _v (M)} = h^{-1}$.

{\it   (iii)} $g_{\vert  T_v M} = \text {Id}$.

\,

Then $V^n$ and $g(M)$  have proportional second fundamental forms and satisfy
all the other assumptions of
 Lemma  \ref {proportional}. Then, by this lemma, $g(M)$, and hence $M$,
 is a Veronese
submanifold.

\end {proof}

\subsection {Coxeter groups and holonomy systems}

\

The goal of this section is to prove Proposition
\ref {holonomy systems}
that will be  important for proving our main
theorems. In order to prove this proposition we need
 some basic results, related to  Coxeter groups,
that we have not found through the mathematical literature. So, and
also for the sake of self-completeness, we include the proofs.

\begin {lema} \label {finite group}
Let $C$ be a Coxeter group acting irreducibly, by linear isometries, on
the Euclidean $n$-dimensional vector space
$(\mathbb V , \langle \, , \, \rangle)$.
Let $H_1, ... , H_r$ be the family of (different)
reflection hyperplanes, associated
to the symmetries of $C$ (that generates $C$). Let us define the group   $G = \{g
\in \text {End}(\mathbb V):\, g \text { permutes }
H_1, ..., H_r \text { and } \text {det}(g) = \pm 1\}.$
 Then $G$ is finite.
\end {lema}

\begin {proof} Let $P_r$ be the (finite) group of bijections  of
 the set $\{1, ... , r\}$.
Let $\rho : G \to P_r$ be the group morphism defined by
$\rho (g) (i) = j$, if $g (H_i) = H_j$. The group $G$ is finite
if and only if $\text {ker}(\rho)$ is finite. Let us  prove
that $\text {ker}(\rho)$  is finite. If $g\in \text {ker}(\rho)$ then it induces
the trivial permutation on the family $H_1 , ... , H_r$. Then, its transpose
$g^t$, with respect to $\langle \, , \, \rangle$, induces the trivial
permutation
on the set of lines $L_1 , ... , L_r$,
where $L_i$ is the line which is perpendicular
to $H_i$, $i = 1, ... , r$ (and hence, any vector in any line
 $ L_1 , ... , L_r$
is
an eigenvector  of $g^t$.
Let us define, for $i\neq j$,
the $2$-dimensional
subspace $\mathbb {V}_{i,j} := \text {linear span of } (L_i\cup L_j)$. This
subspace
is called {\it generic} if there exists $k \in \{1, ... , r\}$, $i\neq k \neq
j$
such
that $L_k \subset \mathbb {V}_{i,j}$. In other words, $\mathbb {V}_{i,j}$ is
generic
if there are at least three different  lines of $\{ L_1 , ... , L_r\}$ which are
contained in   $\mathbb {V}_{i,j}$.
We have, if $\mathbb {V}_{i,j}$ is generic,
 that $g^t : \mathbb {V}_{i,j} \to \mathbb {V}_{i,j}$ is a scalar multiple of
the identity $ \text {Id}_{i,j}$  of  ${V}_{i,j}$.
In fact, any vector in  $L_i\cup L_j \cup L_k$ is an
eigenvector of
 $(g^t)_ {\vert {V}_{i,j}}$. Then, since $\text {dim}({V}_{i,j}) = 2$,
$(g^t)_ {\vert {V}_{i,j}} = \lambda \text {Id}_{i,j}$, for some
$\lambda \in \mathbb {R}$.
Let us define the following equivalence relation $\sim $ on
the set $\{ 1, ... , r\}$:
$i\sim i'$ if there exist $ i_1 , ... , i_l \in  \{1, ... , r\}$
 with $i_1 = i$, $i_l = i'$ and such that
$\mathbb {V}_{i_s, i _{s+1}}$ is generic, for $s=1, ... , l-1$.
Let $i \in \{1, ... , r\}$ be fixed.
 By the previous observations one has that there must exist
$\lambda \in \mathbb {R}$ such that
for any $j\in [i]$ (the equivalence class of $i$) and for any $v_j\in L_j$,
$g^t (v_j) = \lambda v_j$. In order to prove this lemma, it suffices to
show that there is only one equivalence class on $\{1, ... , r\}$.
In fact, if $[i] = \{1, ... , r\}$, then $g^t = \lambda Id$, since
 $L_1, ... , L_r$ span $ \mathbb {V}$ (because of
its othogonal complement
 is point-wise fixed by $C$).
 So
$g = \lambda Id$.
But
$\text {det}(g) =   \pm 1$. Then $\lambda ^ n  = \pm 1$ and hence
$\lambda = \pm 1$.  So, $g = \pm Id$ and therefore there are at most
two elements in $\text {ker} (\rho)$.

Let $i \in \{1, ... ,r\}$ be fixed. Let us show that $[i] =
\{1, ... ,r\}$. If $j\notin [i]$  then $L_j$ is perpendicular to
any $L_k$, for all $k \in [i]$. In fact,
assume that  this is not true for some $k \in [i]$.
Let $s_j \in C$
be
the symmetry across the hyperplane $H_j$. Then
$s_j (L_k)$ is a line, which belongs  to  $\{L_1, ... , L_r\}$,  that
 is contained
in $\mathbb V _{k,j}$ and it  is different
from both $L_k$ and $L_j$. Then $j\sim k$ and therefore $j \sim i$.
A contradiction.
Then, if $j\notin [i]$, $L_k \subset H_j$, for all $k\in [i]$.
So, $s_j$ acts trivially on $\mathbb V _{[i]}$, the subspace spanned
by $\bigcup _{k\in [i]} L_k$. Observe
that $s_j$ commutes with $s_k$,
for all $k\in [i]$. Let now $\mathbb {V} _0$ be the maximal subspace of
$\mathbb {V}$ such that it is point-wise fixed  by all the symmetries
$s_j$ with $j\notin [i]$. Observe that this space is not the null subspace,
since $\mathbb V _{[i]} \subset \mathbb {V}_0$.
If there exists
$j\notin [i]$, then $\mathbb {V}_0 $ must be a proper subspace of $\mathbb {V}$,
since $s_j\neq \text {Id}$.
 On the other hand, if $k\in [i]$, then
$s_k (\mathbb {V}_0 ) \subset \mathbb {V}_0 $, since
$s_k$ commutes with all the symmetries $s_j$, $j\notin [i]$. Then
$\mathbb {V}_0$ is a proper and non-trivial subspace of $\mathbb {V}$ which is
invariant under the irreducible Coxeter group $C$. A contradiction. So, $[i]
=\{ 1, ... , r\}$.

\end {proof}

\begin {lema} \label {permutes hyperplanes}. We are under the assumptions and
notation of the above lemma. Then $G$ acts by isometries.
\end {lema}

\begin {proof}
 By the above lemma, $G$ is finite. By averaging the inner product
$\langle \, , \, \rangle$
over the elements of $G$, we obtain a $G$-invariant inner product $(\, , \, )$
on $\mathbb V$. Since $C\subset G$, then $(\, , \, )$ is $C$-invariant.
Since $C$ acts irreducible, $\langle \, , \, \rangle$ must be proportional to
$(\, , \, )$. Then $G$ acts by isometries on
$(\mathbb {V} , \langle \, , \, \rangle)$.
\end {proof}

\begin {cor} \label {same hyperplanes} Let
$(\mathbb {V}_i , \langle \, , \, \rangle _i)$ be a Euclidean vector spaces
and let $C_i$ be a Coxeter group acting  irreducibly, by
linear isometries, on $(\mathbb {V}_i , \langle \, , \, \rangle _i)$,
$i=1,2$. Let
$\ \ \ \ \ \ \ \ \ \ \ \ \ \ \ \ h: \mathbb {V} _1 \to \mathbb {V} _2$ be a linear
map such that it induces a bijection from the family of reflection
hyperplanes of
$C_1$ into the  family of reflection hyperplanes of $C_2$. Then
$h$ is a homothetical map.
\end {cor}

\begin {proof}

Let $(\, , \, ) = h^* (\langle \, , \, \rangle _2)$ and let
$C^* = h^* (C_2) = h^{-1} C_2 h$.
Observe that the determinant of any element of $C_2$ is $\pm 1$, since
it is an isometry of
$(\mathbb {V}_2 , \langle \, , \, \rangle _2)$. So, any element in
$C^*$ has determinant $\pm 1$. From the assumptions, we obtain that the
family of reflection hyperplanes of the irreducible Coxeter  group $C^*$ of
$(\mathbb {V}_1, (\, , \, ))$  coincides with the
family $H_1, ... , H_r$ of reflection hyperplanes of $C_1$. Then any element
of $C^*$ induces a permutation in this family of hyperplanes.
Then, by Lemma \ref {permutes hyperplanes}, $C^*$ acts by isometries on
$(\mathbb {V}_1 , \langle \, , \, \rangle _1)$. Since $C^*$ acts irreducibly,
one has that $\langle \, , \, \rangle _1$ is proportional to $(\, , \, )$.
This implies that $h$ is a homothecy

\end {proof}

\begin {prop} \label {holonomy systems}

Let $(\mathbb V  , R , K)$ and $(\mathbb V '  , R' , K')$ be
irreducible,  non-transitive (and hence symmetric) holonomy systems.
Let $h: \mathbb V \to \mathbb V '$ be a linear isomorphism such that,
for any $K$-orbit $K.v$ in $\mathbb V$, $h(\nu _v (K.v)) =
\nu _{h(v)} (K'.h(v))$, where $\nu $ denotes the normal space.
Then $h$ is a homothecy and $h^{-1}_*(K') = K$.

\end {prop}

\begin {proof} Observe that the groups $K$ and $K'$
act as irreducible $s$-representations.
We have that $K.v$ is a maximal dimensional orbit
 if and only if $K'.h(v)$ is so.

Recall that, for $s$-representations, an orbit is
maximal dimensional if and only if it is principal.

Let $K.v$ be a principal $K$-orbit. This orbit is an irreducible
(homogeneous) isoparametric submanifold of $\mathbb V$.
There is an irreducible  Coxeter
group  $C$,
associated to this isoparametric submanifold,
that acts on the normal space
$\nu _v (K.v)$ \cite {Te, PT, BCO}.
If $H_1, ... , H_r$ are the reflection hyperplanes
of the symmetries of $C$, then
$$\ \ \ \ \ \ \bigcup _{i=1}^{r} H_i =
\{z\in  \nu _v (K.v): K.z \text
{ is a singular orbit}\} \ \ \ \ \ \ \ \ \
 \text {(a)\ }\ \ \ \ \ \ $$
If $v' = h (v)$ one has the similar objects
$K'.v'$, $\nu _{v'}(K'v')$, $C'$ and $H'_1 , ... , H'_s$
and
$$\bigcup _{i=1}^{s} H'_i =
\{z'\in  \nu _{v'} (K'.v'): K'.z' \text
{ is a singular orbit}\} \ \ \ \   \text {(b)} \ $$
 Moreover, from (a) and (b), one has
that $h$ maps, bijectively, the family $H_1, ... ,H_r$
onto the family $H'_1, ... ,H'_s$. Then, $s=r$ and so we
may assume
that $h(H_i)= H'_i$, $i=1, ... , s$.

Then, from Corollary \ref {same hyperplanes}, one has
that $$h: \nu _w (K.w) \to \nu _{w'}(K'.w')$$
is a homothecy,
for any principal $K$-vector $w$, where $w'= h(w)$.
Denote by $\lambda (w) > 0 $ the homothecy constant of this map.

Observe, since $w\in \nu _w(K.w)$ and $w'\in \nu _{w'}(K'.w')$, that
$$\langle h(w'), h(w')\rangle ' = \lambda (w) \langle w, w\rangle$$
where $\langle \, , \, \rangle$ and $\langle \, , \, \rangle ' $
are the inner products on $\mathbb {V}$ and $\mathbb {V}'$, respectively.

Let  $v_0$ be a fixed $K$-principal vector and let
$M = K.v_0$.

Let $TM = E_1 \oplus ... \oplus E_r$, where $E_1, ... , E_r$ are
 the (autoparallel) eigendistributions of $TM$ associated to
the commuting family of shape operators $A_\xi$ of the isoparametric
submanifold $M\subset \mathbb V$. Associated to any $E_i$ there
is a parallel normal field $\eta _i$, a so-called curvature normal,
such that, for any normal field $\xi$,
 $$A_{\xi \vert E_i} =
\langle \xi , \eta _i\rangle \text {Id}_ {E_i}$$

Let, for $q\in M$,
$S_i(q)$ denote the integral manifold of $E_i$ by $q$.
Such integral manifold  is a so-called curvature sphere.
If  $x\in S_i(q)$ then
$$\nu _x (M) \cap \nu _q (M)= (\eta _ i(q))^\perp $$
where the orthogonal complement is inside $\nu _q (M)$.
Observe that this intersection is non-trivial, since
the codimension of $M$ in $\mathbb V$ is at least $2$.
This implies $\lambda (x) = \lambda (q)$. Since the
eigendistributions span $TM$,  one has that moving along
different curvature sphere one can reach, from $v_0$,
 any other point of $M$. Then $\lambda (x) = \lambda (v_0)$,
for all $x\in M$.

 Observe now that, for any $y\in \mathbb V$, there
exists $\bar {x}\in M$ such that $y\in \nu _{\bar {x}} (M)$.
In fact, such an $\bar {x}$ can
be chosen as a point where the function, from $M$ into
$\mathbb {R}$,
$x \to \langle x, y\rangle$ attains
a maximum.

Then $$\frac {\langle h(y), h(y) \rangle'}{\langle y, y \rangle} =
\frac {\langle h(\bar {x}),
h(\bar {x}) \rangle '}{\langle \bar {x}, \bar {x} \rangle}
= \lambda (x) = \lambda (v_0), $$
for all $0\neq y \in \mathbb {V}$. Then $h$ is a homothecy
of constant $\lambda : = \lambda (v_0)$. This proves the
first assertion.

Let $g\in K'$ and let $T= h^{-1} \circ g \circ h$. Since $h$ is
a homothecy, $T \in \text {SO}(\mathbb {V})$. Then, from the assumptions
and Remark \ref {+simons} one has that $T$ satisfies hypothesis of
Lemma \ref {T(R)}. Then, by this lemma, $T(R)=R$ . This implies, since
 the Lie algebra
of $K$ is generated by $\{R_{x,y}\}$,  that $T$ belongs to
$N(K)$,  the normalizer of $K$ in
$\text {O}(\mathbb {V})$. Moreover, $T$ must belong
 to the connected component
$N_0(K)$ (because of $T$ can be deformed
to the identity, since $K' $ is connected).
 But $N_0(K) = K$, since $K$ acts as an
$s$-representation (see \cite {BCO} Lemma 6.2.2).
Then $T\in K$, thus  $h^{-1}_*(K') = K$.

\end {proof}

\begin {nota}\label {falso}

The above proposition is not true if the holonomy systems are transitive.
In fact, let $(\mathbb V  , R , K)$ and $(\mathbb V '  , R' , K')$ be
the (symmetric)   holonomy systems associated to the  rank $1$
symmetric spaces
$S^{2n} = \text {SO}(2n +1)/\text {SO}(2n)$ and
$\mathbb {C}P^n = \text {SU}(n +1)/ S(U(1)\times \text {U}(n))$, respectively.
In this case  $\text {dim}(\mathbb {V}) = \text {dim}(\mathbb {V}')= 2n$.
Then any linear isomorphism  from  $\mathbb {V}$ into $\mathbb {V}'$, satisfies the
assumption of Proposition \ref {holonomy systems}, since the normal spaces of
non-trivial $K$ or  $K'$-orbits are lines.

\end {nota}

\

\section {non-transitive normal holonomy}

Let $M^n=H.v \subset S ^ {n -1 + \frac {1}{2}n(n+1)}$
be a homogeneous
 submanifold of the sphere.  Assume
 that the (restricted) normal holonomy group,
as a submanifold of the sphere,
acts irreducibly and it is not transitive
(on the unit normal sphere).

 From now on, we will regard $M^n$ as a submanifold
of the Euclidean space $\mathbb {R}^ {n + \frac {1}{2}n(n+1)}$ .
Let $\nu (M)$ be the normal bundle and let
 $\Phi (v)$ be the restricted
normal holonomy group at $v$ (regarding $M$ as a Euclidean submanifold).
Observe that  $\Phi (v)$ acts trivially on
$\mathbb R .v $ and that $\Phi (v)$, restricted to
$\bar {\nu}_v (M): =  \{ v \}^\perp \cap \nu _v(M)$,
is naturally identified with the (restricted)
normal holonomy
group of $M$ at $v$, as a submanifold of the sphere.

 Observe that the irreducibility
of the  normal holonomy group representation on
$\{ v \}^\perp \cap \nu _v(M)$
 implies that $\text {rank}(M) =1$. Namely, $v$ is the only vector of
$\nu _v (M)$ which is fixed by $\Phi (v)$.
 This implies that
 \it $M$ is a full and irreducible
 submanifold of the Euclidean space. \rm
In fact, if $M$ is not full then
  any non-zero constant normal vector is a
parallel normal field which is not a multiple of the position vector.
Then $\text {rank} (M)\geq 2$. A contradiction.
If $M$ is reducible it must be a product
of submanifolds contained in  spheres. Then $\text {rank} (M)\geq 2$.
Also a  contradiction.

One has, from  Remark \ref {correction},
that the first normal space $\nu ^1 (M)$ coincides
with the normal space $\nu M$, regarding $M$ as a Euclidean submanifold.
 This means,
 that the linear map, from $\nu _v (M)$ into
${Sim} (T_vM)$,  $\xi \mapsto
A_\xi $ is injective, where $A$ is the shape operator of $M$.
Since
 $\text {dim}(\nu _v (M)) =
\frac {1}{2}n(n+1) =\text {dim}({Sim} (T_vM))$, then
$A: \nu _v(M) \to  {Sim} (T_vM)$
{\it is a linear  isomorphism}.

Let $\mathcal {R}^\perp _{\xi _1 ,\xi _2}$ be the  adapted
normal curvature tensor
(see Section 1).  This tensor is given by
$$\langle \mathcal {R}^\perp _{\xi _1  ,\xi _2}\xi _3 , \xi _4\rangle
=- \text {trace}([A_{\xi _ 1}, A_{\xi _2 }]\circ [A_{\xi _ 3}, A_{\xi _4 }])
$$
$$=  \langle [A_{\xi _ 1}, A_{\xi _2 }],[A_{\xi _ 3}, A_{\xi _4 }]\rangle =
-\langle [[A_{\xi _ 1}, A_{\xi _2 }],A_{\xi _ 3}], A_{\xi _4 }\rangle$$
Observe that the right hand side of the above equality is, with
the usual identifications,
the Riemannian curvature tensor
$\langle \tilde {R}_{A_{\xi _ 1},A_{\xi _ 2}}
A_{\xi _ 3}, A_{\xi _ 4}\rangle$
of the symmetric space
$\text {Gl}(n) /\text {SO}(n)$.

 Observe that such a symmetric space is
isometric to the following product:
$$ \text {Gl}(n) /\text {SO} (n) =
\mathbb R \times \text {Sl}(n) /\text {SO} (n)$$
The tangent space of the second factor is canonically identified
with the traceless symmetric matrices $Sim _0 (n)$.

Let us consider the so-called {\it traceless shape operator} $\tilde A$
of $M$. Namely,

$$\tilde A_\xi : = A_\xi - \frac 1 n \text {trace}(A_\xi) \text {Id}
= A_\xi - \frac 1 n \langle \xi , H\rangle \text {Id} $$
where $H$ is the mean curvature vector.

Observe that
$$\langle \mathcal {R}^\perp _{\xi _1  ,\xi _2}\xi _3 , \xi _4\rangle =
 \langle [\tilde A_{\xi _ 1}, \tilde A_{\xi _2 }],
[\tilde A_{\xi _ 3},\tilde A_{\xi _4 }]\rangle \ \ \ \ $$
$$ \ \ \ \ \ =  \langle \tilde {R}_{A_{\xi _ 1},A_{\xi _ 2}}
A_{\xi _ 3}, A_{\xi _ 4}\rangle
 =
\langle  {R}_{\tilde A_{\xi _ 1},\tilde A_{\xi _ 2}}
\tilde A_{\xi _ 3},\tilde A_{\xi _ 4}\rangle
\ \ \ \ \ \ \text {(****)}$$
where $R$  is the curvature tensor  at $[e]$
of the  symmetric space $\text {Sl}(T_vM) /\text {SO} (T_vM)$
(see formula (***) of Section 1.3).

If $\bar \nu _v (M) =  \{ v \}^\perp \cap \nu _v (M)$, we
have the following two symmetric non-transitive irreducible
holonomy systems:
$[\bar {\nu}  _v , \mathcal R^\perp , \Phi (v)]
\text { and }
[Sim _0(T_vM), R, \text {SO}(T_vM)]$.

Recall that
for a  symmetric irreducible holonomy system
$[\mathbb V , \bar R, K]$, from Remark \ref {+simons},
the normal space to an orbit $K.v$ is given by
$\nu _v (K.v) = \{ \xi \in \mathbb {V}:
\bar R _{v,\xi} =0\}$

Then, from (****), we have that the map
$\tilde A$ is a liner isomorphism that maps
normal spaces to $\Phi (v)$-orbits into normal spaces to
$\text {SO}(T_vM)$-orbits. Then, by
Proposition \ref {holonomy systems}, $\tilde A$ is a homothecy and
  $\tilde A : \bar {\nu} _v (M) \to {Sim}_0(T_vM)$
transforms  $\Phi (v)$ into $\text {SO}(T_vM)$.
Then $\Phi (v)$ is isomorphic to $\text {SO}(T_vM)$.
Therefore, we have the following result:

\begin {lema} \label {SO(n)}
Let $M^n = K.v \subset S ^ {n -1 + \frac {1}{2}n(n+1)}$
be a
homogeneous submanifold. Assume that the restricted
normal holonomy group of $M$
acts irreducibly and it is non-transitive.
Then  the  representation of the normal holonomy group $\Phi (v)$  on
$\bar {\nu}_v (M)$ is (orthogonally)
equivalent to the isotropy representation of
the symmetric space $\text {\rm Sl}(n)/\text {\rm SO}(n)
\simeq
\text {\rm Sl}(T_vM)/\text {\rm SO}(T_vM)$. Moreover,
the traceless shape operator $\tilde A :
\bar {\nu} _v (M) \to Sim _0(T_vM)$ is a homothecy
that transforms, equivariantly,  $\Phi (v)$ into
$\text {\rm SO}(T_vM)$.
\rm (In  particular,
$\text {dim} (\Phi (v))= \frac {1}{2}n(n-1) =
\text {dim}(\text {\rm SO}(n))$).
\qed
\end {lema}

\begin {prop}\label {constant multiplicities}
Let $M^n = K.v \subset S ^ {n -1 + \frac {1}{2}n(n+1)}$
be a
homogeneous submanifold. Assume that the restricted
normal holonomy group of $M$
acts irreducibly and it is non-transitive.
Then, for any $\xi (t)$ parallel normal section
along a curve, the traceless shape operator $\tilde {A}_{\xi (t)}$
has constant eigenvalues.

\end {prop}

\begin {proof}Note that
$M$ must be full and irreducible as a Euclidean submanifold (see the
beginning of this section).
 Let $p\in M$ be arbitrary and let
 $K_p$ be the isotropy subgroup of $K$ at $p$.
Let us decompose
$$\text {Lie}(K) = \mathfrak {m} \oplus \text {Lie}(K_p)$$
where
$\mathfrak {m}$ is a complementary subspace of $\text {Lie}(K_p)$
Let $B_r(0)$ be an open ball, centered at the origin,
of radius $r$ of $\mathfrak {m} $ such that
$\text {Exp} : B_r(0) \to M$ is a diffeomorphism onto its image
$U = \text {Exp} (B_r(0))$, which is a neighbourhood of $p$
(the inner product on $\text {Lie}(K)$
is irrelevant).

Let  $\beta :[0,1] \to U$ be an arbitrary  piece-wise differentiable
 curve with $\beta (0)= p$. Since $\beta (1) \in U$, there exits
$X\in \mathfrak {m}$ such that
$\beta (1)  = \text {Exp}(X).p$. Let
$\gamma : [0,1] \to M$ be defined by  $\gamma (t) = \text {Exp}(tX).p$.
Let us denote, for $k\in K$,  by  $l_k$ the linear isometry $v \mapsto
k.v$ of $\mathbb {V}$.
Let  $\tau ^\perp _t$ denote the $\nabla ^\perp$-parallel
transport
along $\gamma  _{\vert [0,t]}$.
Then, from remarks 6.2.8 and 6.2.9  of \cite
{BCO},

$$\ \ \ \ \ \ \ \ \ \tau ^\perp _t = (\text {d}l_{\text {Exp}(tX)})_{\vert
\nu _p(M)}\circ \text {e}^{-t\mathcal {A}_X} \ \ \ \
\ \ \ \ \ \ \ \ \ \ \ \ \  \text {(A)}$$
where $\mathcal {A}_X$ belongs to the normal holonomy algebra
$\text {Lie}(\Phi (p))$ and it is defined by
$$\mathcal {A}_X  = \frac {\,\text {d}}{\text {d}t}_{\vert t=0}
\tau ^\perp _{-t}\circ (\text {d}l_{\text {Exp}(tX)})_{\vert
\nu _p(M)}$$

Let $\tau ^\perp _\beta$ be the $\nabla ^\perp$-parallel transport along
$\beta$ and
 $\phi = \tau ^\perp _{-1} \circ \tau ^\perp _\beta$.
Then $\phi$ belongs to $\Phi (p)$, the
restricted normal holonomy group at $p$. In fact, $\phi$ coincides with
 the $\nabla ^\perp$-parallel transport along the null-homotopic, since it is
contained in $U$,  loop
$\beta * \tilde {\gamma}$, obtained from gluing  the curve $\beta$ together with
the curve $\tilde {\gamma}$, where $\tilde {\gamma}(t) = \gamma (1-t)$.

 We have that
$\tau ^\perp _\beta = \tau _{1} \circ \phi $ and so, by (A),
$$\tau ^\perp _\beta =
((\text {d}l_{\text {Exp}(X)})_{\vert \nu _p(M)} \circ
\text {e}^{-\mathcal {A}_X}) \circ  \phi =
\text ({d}l_{\text {Exp}(X)})_ {\vert \nu _p(M)} \circ \bar {\phi}$$
where $\bar {\phi} = \text {e}^{-\mathcal {A}_X}\circ \phi$\,  belongs to
$\Phi (p)$. Then, for any
$\xi \in \nu _p (M)$,

$$\tilde A _{\tau ^\perp _\beta (\xi)} =
\tilde {A} _{\text {d}l_{\text {Exp}(X)} (\bar \phi (\xi))}
= \text {d}l_{\text {Exp}(X)}\circ \tilde {A}_{\bar \phi (\xi)}\circ
(\text {d}l_{\text {Exp}(X)})^{-1}$$
$$ = \text {Exp}(X).\tilde {A}_{\bar \phi (\xi)}
.(\text {Exp}(X))^{-1}
\ \ \ \ \ \ \ \ \ \ \ \ \ \ \ \ \ $$
Then, from the paragraph just before Lemma \ref {SO(n)},
we have that
there exists $g\in \text {SO}(T_p  (M))$ such that
$\tilde {A}_{\bar \phi (\xi)} = g.\tilde {A}_{\bar \phi (\xi)}.g^{-1}$.
Then
$$\tilde A _{\tau ^\perp _\beta (\xi)}
= (\text {Exp}(X).g).\tilde {A}_{\xi}.
(\text {Exp}(X).g)^{-1}$$
This shows that the eigenvalues of $\tilde A _{\tau ^\perp _\beta (\xi)}$
are the same as the eigenvalues of $\tilde  A_{\xi}$.

The curve $\beta$ was assumed to be contained in $U$. Since $p$ is arbitrary,
one obtains that the eigenvalue of $\tilde {A}_{\xi (t)}$ are locally constant
for any $\xi (t)$ parallel normal field along a curve $c(t)$.
This implies that the eigenvalues of $\tilde {A}_{\xi (t)}$ are constant.

\end {proof}

The following lemma is well known and the proof is similar to the case of
hypersurfaces of a space form.

\begin {lema} \label {Dupin Condition}  {\rm (Dupin Condition)}. Let $M$ be a
submanifold of a space of
constant curvature and let $\xi$ be a parallel normal field such that
the eigenvalues  of the shape operator $A_\xi$ have constant multiplicities.
Let
$\lambda : M \to \mathbb R$ be an eigenvalue function of $A_{\xi}$
such that its associated
(and integrable from Codazzi identity)
eigendistribution $E$ has dimension at least $2$. Then $\lambda $
is constant along any integral manifold of $E$ (or equivalently,
$ \text {d} \lambda (E) =  0$).

\end {lema}

\begin {teo} \label {n at least 3}Let $M^n  \subset S ^ {n -1 + \frac 1 2 n(n+1)}$
be a
homogeneous submanifold, where $n> 3$.
Assume that the restricted
normal holonomy group
acts irreducibly and not transitively.
Then $M$ is a Veronese submanifold.
\end {teo}

\begin {proof}

Note that
$M$ must be full and irreducible as a Euclidean submanifold (see the
beginning of this section).

We will regard $M$ as a submanifold of the Euclidean space
$\mathbb {R}^ {n + \frac {1}{2}n(n+1)}$. Then, as we have observed at the
beginning of this section, $A : \nu _p(M) \to  {Sim} (T_pM)$
is an isomorphism ($p\in M$ is arbitrary). Now choose
$\xi \in \nu _p(M)$ such that $A_\xi$ has exactly two eigenvalues
$\lambda _1 (p)$, $\lambda _2 (p)$ with multiplicities $m_1, m_2 \geq 2$
(this is not possible if $n\leq 3$). In particular, we
assume that $m_1 = 2$ and $m_2= n-2$.
We may assume that $\xi$ is small enough such that the holonomy tube
\cite {BCO}
$M_\xi$
is an immersed Euclidean submanifold (see Remark \ref {complete}).
We may also assume that $\xi$ is perpendicular to the
position (normal) vector $p$, since $A_p = -Id$.

There is a natural projection $\pi : M_{\xi} \to M$,
$\pi (c(1) + \bar {\xi}(1)) = c(1)$. Moreover,
$\hat {\xi}$ defines a parallel normal field to $M_{\xi}$,
where $\hat {\xi}(q) = q- \pi (q)$. In this way $M$ is a parallel
focal manifold to $M_\xi$. Namely, $M = (M_{\xi})_{-\hat {\xi}}$.
Observe that the holonomy tube $M_{\xi}$ is not a maximal one and
so it has not a flat normal bundle (this would have been the case, in our
situation,  where
 all of the eigenvalues of $A_{\xi}$ have multiplicity one).
Let $\bar {\xi}(t)$ be a parallel normal field along an arbitrary
curve  $c(t)$ with
$c(0)= p$, \, $\bar {\xi}(0)= \xi$. Then, from Proposition
\ref {constant multiplicities}, the eigenvalues of
the traceless shape operator $\tilde {A}_{\bar {\xi}(t)}$ are
constant and hence the same as the  eigenvalues of
$\tilde {A}_{\xi}$ which are
$\tilde {\lambda}_1 = \lambda _1 (p)
-\frac {1}{n}(2\lambda _1(p) + (n-2)\lambda_2 (p))
$, with multiplicity $2$ and
$\tilde {\lambda}_2
= \lambda _2 (p)-\frac {1}{n}(2\lambda _1 (p) + (n-2)\lambda_2 (p))$,
with multiplicity $n-2$.

Let $H$ be the mean curvature vector field on $M$.
Then the eigenvalues of the shape operator $A_{\bar {\xi}(1)}$
can be written as
 $$ \lambda _ i(c(1)) = \tilde {\lambda} _i + \frac {1}{n}
\langle  \bar {\xi}(1) , H(c(1))\rangle  \ \ \ \ \ i=1,2$$
with multiplicities
$2$ for and $n-2$, respectively (independent of $c(1)\in M$).

From the tube formula \cite {BCO}, one has that the eigenvalues functions
$\hat {\lambda}_1$ and $\hat {\lambda}_2$
of the shape operator $\hat {A}_{\hat {\xi}}$ of the holonomy tube,
restricted to the horizontal
subspace
$\mathcal {H}_q$ of the holonomy tube
$M_\xi$, at a point $q = c(1) + \bar {\xi}(1)$ are:

$$\hat {\lambda}_1 (q) =
\frac {\tilde {\lambda} _1 + \frac {1}{n}\langle  \bar {\xi}(1) , H(c(1))\rangle}
{1-\tilde {\lambda} _1 -
\frac {1}{n}\langle  \bar {\xi}(1) , H(c(1))\rangle}$$
and
$$\hat {\lambda}_2 (q) =
\frac {\tilde {\lambda} _2 + \frac {1}{n}\langle  \bar {\xi}(1) , H(c(1))\rangle}
{1-\tilde {\lambda} _2 -
\frac {1}{n}\langle  \bar {\xi}(1) , H(c(1))\rangle}$$

or, equivalently,

$$\hat {\lambda}_1(q)= \frac {\tilde {\lambda} _1 + \frac {1}{n}
\langle  \hat  {\xi}(q) , H(\pi (q))\rangle}
{1-\tilde {\lambda}_1 - \frac {1}{n}
\langle  \hat {\xi}(q) , H(\pi (q))\rangle}$$
and
$$\hat {\lambda}_2(q)=\frac {\tilde {\lambda} _2 + \frac {1}{n}
\langle  \hat  {\xi}(q) , H(\pi (q))\rangle}
{1-\tilde {\lambda} _2 - \frac {1}{n}
\langle  \hat {\xi}(q) , H(\pi (q))\rangle}$$
with (constant) multiplicities $2$ and $n-2$, respectively.
Observe that $\hat {A}_{\hat {\xi}(q)}$, restricted to the
vertical distribution (tangent to the orbits in $M_\xi$
of the normal holonomy group of $M$ at projected points)
is minus the identity. So, $\hat {A}_{\hat {\xi}(q)}$ has a third eigenvalue
$\hat {\lambda}_3 (q) = -1$ with constant multiplicity $m_3
= \text {dim}(M_\xi) - \text {dim}(M)$.

The real injective function  $s\overset{f} {\mapsto}
\frac {s}{1+s}$ transforms
$\hat {\lambda}_i(q)$ into
$\tilde {\lambda} _i + \frac {1}{n}
\langle  \hat {\xi}(q) , H(\pi (q))\rangle$
($i=1,2$). Then,  $$\hat {\lambda}_1(q) = \hat {\lambda}_1(q')
\iff
 \hat {\lambda}_2(q) = \hat {\lambda}_2(q')\ \ \ \ \ \ \text {(I)}$$
In fact, any of both equalities implies
 $\frac {1}{n}
\langle  \hat {\xi}(q) , H(\pi (q))\rangle = \frac {1}{n}
\langle  \hat {\xi}(q') , H(\pi (q'))\rangle$. This, by the above
equalities, implies (I).

Let now $E_1$ and $E_2$ be the (horizontal) eigendistributions
associated to eigenvalue functions $\hat {\lambda}_1$ and
$\hat {\lambda}_2$ of the shape operator $\hat {A}_{\hat {\xi}}$.
Observe that $\text {dim}(E_1)=2$ and $\text {dim}(E_2)= n-2\geq 2$.

\

\it Up to here everything is valid, except the last inequality,
 also for
$n=3$. \rm \ \ \ \ \ (II)

\noindent (This will be used in next section where we deal
with the case $n=3$).

\

If $\gamma (t)$ is a curve that lies in $E_1$ then, from the Dupin Condition
(see Lemma \ref {Dupin Condition}) we have that $\hat {\lambda} _1$ is constant
along $\gamma$. So, by (I), $\hat {\lambda} _2$ is also constant along
$\gamma$.
The same is true if $\gamma$ lies in $E_2$. This implies that
$0 = v(\hat {\lambda}_1) = v(\hat {\lambda}_2) = v(\hat {\lambda}_3)$
for any vector $v$ that lies in $\mathcal {H}$.
Then the eigenvalues of the shape  operator $\hat {A}_{\hat {\xi}}$
are constant
along any horizontal curve. Since  any two points, in a holonomy tube,
can
be joined by a horizontal curve we conclude
that the (three) eigenvalues of
$\hat {A}_{\hat {\xi}}$ are constant on $M_\xi$.

Then $\hat {\xi}$ is a parallel isoparametric (non-umbilical) normal section.
Observe that  $M_{\xi}$ is a full irreducible Euclidean submanifold, since
$M$ is so. Moreover, $M_{\xi}$ is complete with
the induced metric (see Remark \ref {complete}).
Then, by \cite {BCO},\cite {DO}, $M_\xi$ must be a submanifold
with
constant principal curvatures. Since $M=(M_\xi)_{-\hat {\xi}}$, we have
that $M$ is also a submanifold with constant principal curvatures.
Any principal holonomy tube of $M$ has codimension at
least $3$ in the Euclidean space, since
 the normal holonomy of $M$, as a submanifold of the sphere,  is non-transitive.
Then,
by the theorem of Thorbergsson \cite {Th,O2, BCO}, $M$ is an
orbit
of an (irreducible) $s$-representation.

The fact that
 $M$ is a Veronese submanifold follows from Proposition \ref {n-Veronese}.

\end {proof}

\,

\begin {nota}\label {complete}

Let $M^n=H.v$ be a full  irreducible
 homogeneous submanifold  of $\mathbb {R}^N$ which
is (properly) contained in the sphere $S^{N-1}$.
We are not assuming that $M$ is compact
(in which case the assertions of this remark are trivial).

By making use of the homogeneity of $M$ one obtains that
there exists $\varepsilon >0$ such that:  if $\xi \in \nu (M)$ with
$0< \Vert \xi \Vert < \varepsilon$ then any  of the  eigenvalues $\lambda$ of
the shape operator $A_{\xi}$ satisfies $\vert \lambda \vert < 1-a$, for some
$0<a<1$.

Let us assume that $\text {rank} (M) =1$, i.e.,
$M$ is not a submanifold of higher rank
(otherwise, $M$ would be
 an orbit of an $s$-representation and hence compact).

Let $ \xi \in \nu _v (M)$ with $0< \Vert \xi \Vert < \varepsilon$
and let us consider the normal holonomy subbundle by $\xi$
\cite {BCO} of the normal bundle $\pi :\nu (M) \to M$.
$$\text {Hol}_{\xi} (M) = \{\eta \in \nu (M): \eta \overset {\mathcal {H}}{\sim}
\xi\}$$
where $\mathcal {H}$ is the horizontal
distribution of $\nu (M)$ and
$\eta  \overset {\mathcal {H}}{\sim} \xi$ if $\eta $ and $\xi$
can be joined by
a horizontal curve. Equivalently,
$\eta  \overset {\mathcal {H}}{\sim} \xi$ y
$\eta $ is the $\nabla ^{\perp}$-parallel transport of $\xi$ along
some curve.

One has that the fibres of $\pi : \text {Hol}_{\xi} (M)\to M$ are
compact. In fact,
$\pi ^{-1} (\{\pi (\eta)\}) = \Phi (\pi (\eta)).\eta $,
where $\Phi$ denotes the normal holonomy group. Observe that such
 a group is compact, since its connected component acts as an
 $s$-representation (see the discussion inside the proof of
Theorem \ref {3-Veronese}, Case (2), (c)).

Let us consider the normal exponential map
$\text {exp}^{\nu}: \nu (M) \to \mathbb {R}^N$, given by
$\text {exp}^{\nu}(\eta) = \pi (\eta) + \eta$. Let $\eta \in \nu _p(M)$
and identify, as usual, via $\text {d}\, \pi$,
$T_pM \simeq \mathcal {H}_\eta$. The vertical distribution
$\nu _{\eta} =
T_{\eta}\nu _p (M)$ is canonically identified to $\nu _p (M)$. With this
identification one has the well-known expression for the differential of
the normal exponential map:
$$\text{d}  (\text {exp}^{\nu})\, _{\vert \mathcal {H}_\eta}= (I-A_{\eta})\, ,
\ \ \ \ \ \
\text{d}  (\text {exp}^{\nu})\, _{\vert \nu_\eta} = Id_{\nu _p (M)}
\ \ \ \ \ \text {(C)}$$

Then $\text {exp}^{\nu} : \text {Hol}_{\xi} (M) \to \mathbb {R}^N$
is an immersion.
The image  of this map is the so-called holonomy tube $M_{\xi}$ of
$M$ by $\xi$. It is given by
$$M_{\xi} =
  \{ c(1) + \bar {\xi} (1):\bar {\xi}(t) \text { is }
\nabla^\perp \text {-parallel
 along } c(t) \text { where } c(0) = p, \, \bar {\xi}(0) =
\xi   \}   $$

 Many times, and in particular in the proof of Theorem \ref {n at least 3},
for the sake
of simplifying the notation, the immersed submanifold
$\text {exp}^{\nu} : \text {Hol}_{\xi} (M) \to \mathbb {R}^N$
will be also denoted by $M_{\xi}$.

One has that the Euclidean submanifold
 $\text {exp}^{\nu} : \text {Hol}_{\xi} (M) \to \mathbb {R}^N$, with the
induced metric $\langle \, , \, \rangle$, is a complete Riemannian manifold.
In fact, let $(\, , \, )$ be the
Sasaki metric on
$\text {Hol}_{\xi} (M)$. In such a metric the horizontal distribution
is perpendicular to the vertical one. Moreover, $\pi$ is a Riemannian submersion
and the metric in the vertical space
$\Phi (p).\eta$ is that induced from the metric on the  normal
space $\nu _p (M)$. Since $M$ is complete and the fibres are compact,
then $(\, ,\, )$ is complete.
Then, from $(C)$, $a^2(\, ,\, ) \leq \langle \, , \, \rangle$.
This implies that the induced metric is also complete.

\end {nota}

\

\section {The proof of the conjecture in dimension 3}

\begin {teo} \label {3-Veronese}
 Let $M^3 = H.p$ be a $3$-dimensional homogeneous submanifold
of the sphere $S^{N-1}$ which is full and irreducible
(as a submanifold of the
Euclidean space $\mathbb {R}^N$). Assume  that the normal holonomy group
of $M$ is non-transitive. Then $M$ is an orbit of an $s$-representation.

\end {teo}

\begin {proof}

 Assume that $M$ is not isoparametric
(in which case it must be an orbit of an $s$-representation).
Then, by  Lemma \ref {aux},  the normal holonomy of $M$, as a submanifold
of the sphere acts irreducibly  and $N=9 = 3 + \frac {1}{2}
3 (3+1)$. We have also that
the first normal bundle, which coincides
with the normal bundle, has maximal
codimension.

 Keeping the notation and general constructions
in the proof of Theorem
\ref {n at least 3}, we have that everything is still valid up to
  {(II)}.
The only difference is that the eigenvalue
$\hat {\lambda} _2$ has multiplicity
$1$. So, we have the Dupin condition only for the eigendistribution
$E_1$ but not for the $1$-dimensional eigendistribution $E_2$.

Let $\bar M = M_\xi /\mathcal {E}_1$ be the quotient of the
(partial) holonomy tube $M_\xi$ by the (maximal)
integral manifolds of the
$2$-dimensional integrable distribution
$E_1$.

Observe that the (partial) holonomy tube $M_\xi$ has dimension $5$.
In fact, from Lemma  \ref {aux},
any focal orbit  of the restricted normal holonomy group $\Phi (p)\simeq
\text {SO}(3)$ has dimension $2$
(and it is isometric to the Veronese $V^2$).

By \cite {BCO}, Theorem 6.2.4, part (2) one has that
$H \subset \text {SO}(9)$ acts by
(extrinsic) isometries on $M_\xi$. Moreover, the projection
$\pi : M_\xi \to M$ is $H$-equivariant.

If $H.(p + \xi ) = M_\xi$, then $M_\xi$ is a full and irreducible
homogeneous Euclidean submanifold which is of higher rank. Then, in this case,
by the rank rigidity theorem for submanifolds,
$M_\xi$ is an orbit of an $s$-representation. Hence $M = (M_\xi)_{-\hat \xi}$
is an orbit of an
$s$-representation.

So, we may assume that $H.(p + \xi ) \subsetneq M_\xi$.
Let $\mathfrak {h} = \text {Lie}(H)$.
Let us consider the subspace $\mathfrak {h}.(p+\xi)$ of $T_{p+\xi}M_\xi$.
This subspace has dimension at least $3$, since
$\text {d}\pi (\mathfrak {h}.(p+\xi)) = \mathfrak {h}.p = T_pM$.
The horizontal subspace $\mathcal {H}_{(p+\xi)}$ of $T_{p+\xi}M_\xi$ has
dimension $3$. Since $T_{p+\xi}M_\xi$ has dimension $5$,
$\text {dim} (\mathcal {H}_{(p+\xi)}
\cap \mathfrak {h}.(p+\xi)) \geq 1$.

$$\text {{\it Case (1): }} \ \ \ \ \ \ E_1(x) + (\mathcal {H}_{x}
\cap \mathfrak {h}.x) = \mathcal {H}_{x}, \text { for some } x\in M_{\xi}
 \ \ \ \ \ \ \ \ \ \ \ \ \ \ \ \ \ \ \
\ \ \ \ \ \ \ \ $$

We may assume that $x = p + \xi$.
Observe that if   the above equality holds at
$(p + \xi)$ then it also holds for $q$ in some open neighbourhood $U$ of
 $(p + \xi)$ in  $M_\xi$.

Recall, continuing with the notation in
the proof of Theorem  \ref {n at least 3},
that the eigenvalues functions (which are differentiable)
 of the shape operator $\hat {A}_{\hat \xi}$
at $q$ are: $\hat {\lambda}_1 (q)$ with multiplicity $2$,
$\hat {\lambda}_2 (q)$ with multiplicity $1$ and $\hat {\lambda} _3 (q) = -1$
with multiplicity
$2$ (whose associated eigenspace  is the vertical distribution $\nu _q$).

On one hand, from the Dupin condition,
since $\text {dim}(E_1) = 2$, and the equivalence $(I)$ in
the proof of the above mentioned theorem, we have  that
$$0 = v(\hat {\lambda} _1) = v(\hat {\lambda} _2) = v (\hat {\lambda}_3)$$
for any  $v\in E_1(q)$. Or, briefly,
$$\{0\} = E_1(q)(\hat {\lambda} _1) = E_1(q)(\hat {\lambda} _2) =
E_1(q) (\hat {\lambda}_3)$$

On the other hand, if $X\in \mathfrak {h}$,
$$0 = (X.q)(\hat {\lambda} _1) = (X.q)(\hat {\lambda} _2) =
(X.q). (\hat {\lambda}_3)$$
In fact, this follows from the fact that the parallel normal field
$\hat {\xi}$ of $M_\xi$ is $H$-invariant and that $\hat {A}_{h.\hat {\xi} (q)}
= h.\hat {A}_{\hat {\xi}(q)}.h^{-1}$, for all $h\in H$.

Then, from the assumptions of this case,
$$\{0\} = \mathcal {H}_q (\hat {\lambda} _1) =
\mathcal {H}_q(\hat {\lambda} _2) =
\mathcal {H}_q(\hat {\lambda}_3) \ \ \ \ \ \ \ \text {(III)}$$
for any $q\in U$.

Since $M$ is (extrinsically) homogeneous, the local normal holonomy
groups have all the same  dimension. Then the local
 normal holonomy group at any $x \in M$ coincides with the   restricted
normal holonomy group $\Phi (x)$.

The $\nabla ^\perp$-parallel transport along short loops, based at
$p\in M$, produces a neighbourhood $\Omega$ of $e$ in the local normal
holonomy group, see \cite {CO, DO}).
 This implies, from
(III), that the eigenvalues of $\hat {A}_{\hat {\xi}(p + \omega .\xi)} $
are the same as the eigenvalues   $\hat {\lambda}_1(p + \xi)$,
$\hat {\lambda}_2(p + \xi)$, $\hat {\lambda}_3(p + \xi) = -1$
of $\hat {A}_{\hat {\xi}(p + \xi)}$, for all $\omega \in \Omega$.
From this it is  standard to show  that the eigenvalues of
$\hat {A}_{\hat {\xi}(p + \phi .\xi)}$ are the same of those of
$\hat {A}_{\hat {\xi}(p + \xi)}$, for all
$\phi \in \Phi (p)$. Therefore, the eigenvalues of $\hat {A}_{\hat {\xi}}$
are constant on $p + \Phi (p).\xi = \pi ^{-1} (\{p\})$.
Since $H$ acts
transitively on $M$, then $H. \pi ^{-1} (\{p\}) = M_\xi$.
This implies,
since $\hat {\xi}$ is $H$-invariant,
that the eigenvalues
of $\hat {A}_{\hat {\xi}}$ are constant on $M_\xi$.

Observe that  the parallel normal
field $\hat {\xi}$ is not umbilical, since $\hat {A}_{\hat {\xi}}$
 has three distinct (constant)
eigenvalues. Then, from  \cite {DO} (see Theorem 5.5.8 of \cite {BCO}),
$M_{\xi}$ has constant principal curvatures. So,
$M = (M_{\xi})_{-\hat {\xi}}$
has constant principal curvatures. If $\tilde M$ is a principal holonomy
tube of $M$, then $\tilde M$ is isoparametric  \cite {HOT}. Observe that
 $\tilde M$ is not a
hypersurface of a sphere (since the normal holonomy group, in the Euclidean space,
is not transitive on the orthogonal complement of the position vector), then
by the theorem of Thorbergsson \cite {Th, O2} $\tilde M$ is an orbit
of an $s$-representation. Then   $M$ is an orbit of an $s$-representation, since
it is a focal (parallel) manifold to $\tilde {M}$.

\,

$$\text {\it Case (2): } \ \ \ \ \ \ E_1(x) + (\mathcal {H}_{x}
\cap \mathfrak {h}.x) \subsetneq \mathcal {H}_{x}, \text { for all }
x\in M_{\xi} \ \ \ \ \ \ \ \ \ \ \ \ \ \ \ \ \ \ \ \ \ \ \ $$
or equivalently, $ (\mathcal {H}_{x}
\cap \mathfrak {h}.x) \subset E_1(x)$, since $\text {dim} (E_1 (x)) =2$
and $\text {dim} (\mathcal {H}_{x}) =3$.

This case
splits into several  sub-cases, depending
on how big is the group $H$. Namely, depending on
$\text {dim} (H)\geq 3 = \text {dim}(M)$.
The most difficult case is the generic one where
$\text {dim} (H)= 3$. For  this case we will have to use  topological
arguments.

 Note  that $\text {dim}(H) \leq 6$.
In fact, $H$ acts effectively on $M$, since $M$ is a full submanifold.
Otherwise, if $h\in H$ acts trivially on $M$ then it acts trivially on
the  (affine) span of $M$ which is $\mathbb {R}^9$. But the dimension
of the  isometry group
of an $n$-dimensional Riemannian manifold is bounded by
$\frac {1}{2}(n+1)n$ (the dimension of the isometry group of an $n$-dimensional
space of constant curvature). In our case, since $n=3$,
$\text {dim} (H)\leq 6$.

Observe that $H$ cannot be abelian. In fact if $H$ is abelian,
since the dimension
of the ambient space $N =9$ is odd, the  (connected)
subgroup $H\subset \text {SO}(9)$ must fix a vector, let us say $v\neq 0$.
So, no $H$-orbit  $H.q$  is a full submanifold, since it is contained
in $q + \{v\}^\perp$. A contradiction, since  $M = H.p$ is full.

Observe that $\text {dim} (H)$ cannot be $5$. In fact, if
$\text {dim} (H) = 5$ then the isotropy $H_p$ has dimension $2$
and so it is abelian. We regard
$H_p \subset \text {SO}(T_pM)\simeq \text {SO}(3)$,
via the isotropy representation. But the rank of $\text {SO}(3)$ is
$1$ and so it has no abelian two dimensional subgroups. A contradiction.

\

({\bf a}) $\text {dim} (H) = 6$.

In this case we must have
that $(H_p)_0 = \text {SO}(3)$, since $\text {dim} (H_p) = 3$.
Since $\text {SO}(3)$ is simple, the slice representation $sr$ of
$(H_p)_0$ on the normal space $\nu _p (M)$
 must be either trivial or its image has dimension $3$.
In the first case we obtain that all shape operators $A_\mu$ of $M$
at $p$ are a multiple of the identity, since they commute all
with $(H_p)_0$. Note that
$A _{\mu} = A_{h.\mu} = h.A _{\mu}. h^{-1}$. So $M=M^3$ is an umbilical submanifold
of $ S^8\subset \mathbb {R}^9$. So, $M$ is not full. A contradiction.

Let us deal with the case that the image of the slice representation
has dimension $3$. By \cite {BCO}, Corollary 6.2.6  $sr ((H_p)_0)\subset \Phi (p)$
where $\Phi (p)$ is the restricted normal holonomy group of $M$ as a Euclidean
submanifold. Since
$\text {dim} ({\Phi (p)}) =3$, we conclude that $sr ((H_p)_0)= \Phi (p)$.
Then, any holonomy
tube of $M$ is an $H$ orbit. In particular the principal ones, which have flat
normal bundle. But the holonomy tubes are full and irreducible Euclidean submanifolds,
which have codimension at least $3$ (since $\Phi (p)$ acts on the $6$-dimensional
normal space $\nu _p (M)$
with cohomogeneity $3$). Then, by the theorem of Thorbergsson \cite {Th, O2},
any holonomy tube is an orbit of an $s$-representation and so $M$ is an orbit
of an $s$-representation. By Proposition  \ref {n-Veronese} one has that
$M = M^3$ is a Veronese submanifold.

\

 ({\bf b}) $\text {dim} (H) = 4$.

In this case the isotropy $H_p$ has dimension
$1$. If the slice representation $sr$ of $(H_p)_0$ is trivial,
then, as in (a), all shape operators at $p$ commute with
$(H_p)_0 \simeq S^1$. A contradiction, since the family of shape
operators is $Sim (T_pM)$.

Let us then restrict to the case that the slice representation is not trivial.
For this we have to use a result of \cite {OS}
 (see \cite {BCO}, Theorem 6.2.7). In fact, we need the following
weaker version, which was the main step in the proof
of Simons holonomy theorem given in \cite {O5}. Namely,
Proposition 2.4 of  \cite {O5}: \it
for a full and irreducible
 $H$-homogeneous Euclidean submanifold $M^ n $, $n\geq 2$,
  the projections, on the normal space
$\nu _p (M)$, of the  (Euclidean) Killing fields
given by the elements of $\mathfrak {h} = \text {Lie} (H)$, belong to
the normal holonomy algebra $\mathfrak {g}$.
\rm

Then, in our situation, since $\text {dim}(\mathfrak {h}) = 4$ and
$\text {dim}(\mathfrak {g}) = 3$, there must exist $0\neq X\in \mathfrak {h}$
such that it projects trivially on the normal space. Such an $X$ cannot
be in the isotropy algebra, since we assume that the slice representation
of $(H_p)_0 \simeq S^1$ is non-trivial. This implies that
$0\neq X.p \in T_pM$.

Let us consider the $H$-invariant
parallel normal field $\hat {\xi}$ of $M_\xi$. Recall that
$(M_\xi)_{-\hat \xi} = M$ (and so $M$ is a parallel focal manifold
of $M_{\xi}$).

Since $X$ projects trivially on $\nu _p(M)$,
$X.q \in \mathcal {H}_q$,  for all
$q \in (p + \Phi (p).\xi)
= ((\pi)^{-1}( \{ p \} ))_q \subset M_\xi$.

Recall that we are in  {\it Case (2)}. Then, $X.q \in
E_1(q)$, for all $q \in (p + \Phi (p).\xi)$.
Let us consider the curve $\gamma (t) = \text {Exp}(tX).p$
of $M^3$. One has that $\gamma ' (0) = X.p \neq 0$.
Let $q \in (p +\Phi (p).\xi)$ and let   $\psi (t)$
be the normal parallel transport of
$(q-p) \in \nu _p (M)$ along $\gamma (t)$.
Then $\psi (t) =  \hat {\xi} (\gamma (t) + \psi (t))$,
as it is well known, from the construction of holonomy
tubes  \cite {HOT, BCO} (observe that $M_\xi = M_{q-p}$).
From the tube formula of \cite {BCO}, Lemma 4.4.7 (the notation in this lemma
permutes our objects),
$$A_ {(q-p)} =
\hat {A}_{(q-p)\vert \mathcal {H}}.
((Id - \hat {A}_{(q-p)})_{\vert \mathcal {H}})^{-1}$$
one has that $E_1 (q)$ is an eigenspace
of the shape operator $A_ {(q-p)}$ of $M$.

On the one hand,  since $\pi (q) = q  -\hat {\xi}(q)$,

$$\text {d} \pi ( E_1 (q))
= (Id + \hat {A}_{\hat {\xi}})( E_1 (q))\subset  E_1 (q)$$

On the other hand, since $\hat {\xi}$ is $H$-invariant and
$\hat {\xi}(q) = (q-p)$,

$$\text {d}\pi (X.q) = \frac {\text { d}}{\text {d}t}\vert _ 0
(\text {Exp} (tX). q - \hat {\xi} (\text {Exp} (tX).q)) \ \ \ \ \ \ \ \ $$
$$ = \frac {\text { d}}{\text {d}t}\vert _0
(\text {Exp} (tX). q - \text {Exp} (tX).(q -p)) = X.p \ \ \
 \ \ \ \ \ \ $$

Therefore, $X.p$ belongs to an eigenspace  of any shape operator
$A_{q-p}$ of $M$, such that $q\in  (p + \Phi (p).\xi)$
(recall that we have assumed, without loss
of generality,  that $\xi $ is perpendicular to the position
vector $p$).

Observe that  $\Phi (p).\xi$ spans $\{p\}^\perp$, since
$\Phi (p)$ acts irreducibly on $\{p\}^\perp$.
So $X.p$ is an eigenvector of any shape operator $A_\eta$,
where $\langle \eta ,p\rangle =0$.

Since
$A_p = -Id$, we conclude that $X.p$ is an eigenvector of
all shape operators of $M$ at $p$. This is a contradiction,
since the family of shape operators at $p$ coincides with
$Sim (T_pM)$.

\

({\bf c}) $\text {dim} (H) = 3$.

Since we have  excluded the case where $H$ is
abelian, then $H$ must be  simple, with  universal
cover the (compact) group $\text {Spin}(3)\simeq S^3$.
This case is the generic
one where the isotropy is finite. Note that $M$ must
be compact.

Also note that the  (full) normal holonomy group
$\tilde {\Phi}(p)$ of $M$
is compact. In fact,
 $(\tilde {\Phi}(p))_0$ coincides with
 the restricted normal holonomy group $\Phi(p)$.
Moreover, $\tilde {\Phi}(p)$ is included in  the compact group
$N(\Phi (p))$, the normalizer
 of ${\Phi}(p)$ in $\text {O}(\nu_p(M))$. Observe that
$(N(\Phi (p)))_0 = \Phi (p)$, since $\Phi (p)$ acts as an
$s$-representation (see \cite {BCO} Lemma 6.2.2). Then
$\tilde {\Phi}(p)$ has a finite number of connected components, as well
as $\tilde {\Phi}(p).\xi$.
This implies that $M_{\xi}$ is compact.

Let us construct the so-called {\it caustic fibration}.
The eigenvalues functions of  $\hat {A}_{\hat {\xi}}$
are bounded
on $M_{\xi}$.
Since $M$ is contained in a sphere, $M_{\xi}$ is contained in a
(different) sphere. If $\eta $ is the position vector field
of $M_{\xi}$, then $\eta $ is an umbilical parallel normal field.
In fact,
 $\hat A _\eta = - \text {Id}$. By adding, eventually,
to the parallel normal field
 $\hat {\xi}$ a (big) constant multiple of $-\eta $ we obtain a new parallel
 and $H$-invariant normal field, such that its associated shape
operator has the same
 eigendistributions
 as $\hat A _{\hat \xi}$ and  all of the
three eigenvalues functions are
 everywhere positive and so nowhere vanishing. Just for the sake of simplifying
the notation, we also denote  this perturbed normal field  by $\hat \xi$.
The eigenvalues of $\hat {A}_{\hat {\xi}}$ are also denoted by $\hat {\lambda} _1$,
$\hat {\lambda} _2$, $\hat {\lambda} _3$, which differ from the original
ones by a (same) constant $c$.

The {\it caustic} map $\rho$, from $M_\xi$ into $\mathbb {R}^9$,
$q \overset {\rho} {\mapsto} q + (\hat {\lambda} _1(q))^{-1}\hat {\xi}(q)$ has
constant rank. In fact, $ {ker}(\text {d} \rho)= E_1$ has constant dimension
 $2$, since from the Dupin condition,  $\hat {\lambda} _1$
is constant along any integral manifold $Q_1(q)$ of $E_1$. Observe that
$\hat {\lambda}_2$ is also constant along $Q_1(q)$, due to equivalence (I) in the
proof of Theorem \ref {n at least 3} (and the same is true, of course, for the
third eigenvalue  $\hat {\lambda}_3 \equiv -1 + c$).

Let $\bar M = M_\xi /\mathcal {E}_1$ be the quotient
of $M_\xi$ by the family $\mathcal {E}_1 $ of (maximal) integral
manifolds of $E_1$.  From Lemma \ref {bundle-like} we have that $\bar M$
is a compact $3$-manifold immersed in $\mathbb {R}^9$, via the projection
$\bar {\rho}$, of the caustic map $\rho$, to
the compact quotient manifold $\bar M$. Moreover, $\bar {\pi}:M_\xi \to  \bar M$
is a fibration, where $\bar {\pi} : M_\xi \to  \bar M$ is the projection.
The distribution $E_1$ is $H$-invariant, since $\hat \xi$ is so.
So, the action of $H$ on $M_\xi$ projects down to an action on $\bar M$.
 So, $\bar  \pi$ is $H$-equivariant.

Observe that $\rho$ is $H$-equivariant, since $\hat \xi$ is $H$-invariant.
Then,  since
$\bar \pi $ is $H$-equivariant, the immersion
$\bar \rho : \bar M \to \mathbb {R}^9$
is $H$-equivariant.

We have the following two $H$-equivariant fibrations on $M_\xi$:

$$\ 0 \stackrel{}{\rightarrow}  {\Phi}(p).\xi \stackrel{}{\rightarrow}M_\xi
\stackrel{\tilde {\pi}}{\rightarrow}
 \tilde {M} \stackrel{}{\rightarrow} 0 \ \ \ \ \ \ \ \
\text {\bf (holonomy tube fibration) }$$

 $$0 \stackrel{}{\rightarrow} Q \stackrel{}{\rightarrow}M_\xi
\stackrel{\bar  {\pi}}{\rightarrow}
 \bar M \stackrel{}{\rightarrow} 0 \ \ \ \ \ \ \ \ \text {\bf (caustic fibration) }
 \ \ \ \ \ $$

\

 \noindent where $Q$ is any integral manifold of $E_1$ and $\tilde {M}$
is the quotient manifold $M_{\xi}$ over the connected component
of the fibres of $\pi : M_{\xi} \to  M$, which are orbits of the restricted
normal holonomy groups $\Phi (p)$, $p\in M$. We have that $\tilde {M}$
is a finite cover of $M$.

{\it Recall that we are under  the  assumptions of
Case (2)}

We will derive  a   {\it topological} contradiction.
This is by using that the holonomy tube $M_\xi$
is the total space of above two different fibrations.

On the one hand the holonomy tube has a finite fundamental group
$\pi _ 1 (M_\xi)$. This follows from the long exact sequence
of homotopies, associated to the holonomy tube fibration. In fact,  the fibres
are (real) projective $2$-spaces (which have a finite fundamental group).
Moreover, the base space $\tilde {M}$ has also a finite fundamental group, since
it is an orbit, with finite isotropy, of the group
$\text {Spin}(3)\simeq S^3$.
 Since the fibres of the caustic fibration are connected and the total space
$M_{\xi}$ has finite fundamental group,
then  the caustic (base) manifold $\bar {M}$ has a finite fundamental group.

On the other hand, from Lemma \ref {caustic-polar} we have that the fundamental
group of the caustic manifold $\bar {M}$ is not finite (this is by showing that
$H$ acts with cohomogeneity $1$  and without singular orbits on $\bar {M}$).

A contradiction. So we can never be under the assumptions of
 Case (2) \ if $H\simeq \text {Spin}(3)$.

\

This finishes the proof   that $M$ is an orbit of an
$s$-representation.
\end {proof}

\begin {lema} \label {aux} We are in the assumptions of Theorem \ref {3-Veronese}.
Then, if $\text {rank} (M) = 1$,
  the (restricted) normal holonomy group $\Phi (p)$,
as a submanifold of the sphere,
 acts irreducibly and $N = 9$. Moreover, the (restricted) normal holonomy
acts as the
action of  $\text {SO}(3)$, by conjugation, on the traceless
$3\times 3$-symmetric matrices. Furthermore, the traceless shape operator
$\tilde A$ of $M$ at $p$ is $\text {SO}(3)$-equivariant.

\end {lema}

\begin {proof}
Let us regard $M^3$ as a submanifold of the Euclidean space $\mathbb
{R}^N$.
If $M$ is not of higher rank one has, from Proposition \ref {cota},
that the   (restricted)
normal holonomy group $\Phi (p)$ acts irreducibly on $\bar {\nu}(p)$
(the orthogonal complement
of the position vector $p$).
Since $\Phi (p)$ is non-transitive (on the unit sphere of
$\bar {\nu}_p(M)$),
the first normal space, as a submanifold of the Euclidean space,
 coincides with the normal space (see Remark \ref
{correction}).
Then, the codimension $k =N-3$  satisfies
$k\leq 6 = \frac {1}{2}3(3+1)$.
Then the normal holonomy group representation
coincides with the isotropy
representation of an irreducible symmetric space
of  rank at least $2$ and dimension at most $5$.
Then, by Remark \ref {5-symmetric}, the normal holonomy
representation is equivalent to the   isotropy
representation of $\text {Sl}(3)/\text {SO}(3)$.
 So the codimension of $M$, in the sphere,
is $5$ and hence $N=9$.  The equivariance follows form
Lemma \ref {SO(n)}.

\end {proof}

\begin {lema}  \label {bundle-like} \rm (Caustic fibration lemma). \it
Let $\hat {M}$ be a compact immersed submanifold of $\mathbb {R}^N$ which is
contained in the sphere $S^{N-1}$. Let
$\hat {\xi} $ be a parallel normal field to $\hat {M}$
such that the eigenvalues
of the shape operator $A_{\hat {\xi}}$ have constant multiplicities on $\hat {M}$.
Let $\hat {\lambda} : \hat {M} \to \mathbb {R}$ be an eigenvalue function of
$A_{\hat {\xi}}$ whose associated (integrable)
eigendistribution $E$ has (constant) dimension
at least $2$. Let $\mathcal {E}$ be the family of (maximal)
integral manifolds of
$E$. Assume that the eigenvalue function $\hat {\lambda}$ never vanishes
{\rm (this can always be assumed by adding to $\hat {\xi}$ an appropriate
constant multiple of the umbilical position vector)}. Then

{\it \ \ (i)} Any integral manifold $Q  \in \mathcal {E}$
is compact.

{\it \  (ii)} The quotient space $\bar {M} = \hat {M}/\mathcal {E}$
is a (compact) manifold and the projection $\pi : \hat {M}
\to \bar {M}$ is a
fibration (in particular, a submersion).

{\it \  (iii)} The caustic map $\rho : \hat {M} \to \mathbb {R}^N$,
 $\rho (q) = q + (\hat {\lambda} (q))^{-1}\hat {\xi}(q)$, projects down
to an immersion $\bar {\rho}: \bar {M} \to \mathbb {R}^N$ (i.e. $\rho =
\bar {\rho}\circ \pi$).

\end {lema}

\begin {proof}

From the Dupin condition, see Lemma \ref {Dupin Condition}, one has
that $\hat {\lambda}$ is constant along any integral manifold
$Q$ of $E$.

Consider the caustic map $\rho (q) = q + (\lambda (q))^{-1}\hat {\xi}(q)$
(see the proof of
Theorem \ref {3-Veronese}, Case (2),(c)). Then $ker (\text {d}\rho) = E$ and
so $\text {d} {\rho}$ has constant rank.
From the local form of a map with constant rank and the compactness
of $\hat {M}$ one has that there exists a finite open cover $V_1, ... , V_d$ of
$\hat {M}$ such that,  for any $i=1, ..., d$ and $q, q'\in V_i$,  the following
equivalence holds:
 $$\rho (q) = \rho (q')  \iff q \text { and } q'
\text { belong both to a same integral manifold of }
E.$$

This implies that any (maximal) integral manifold $Q$ of $E$ must be a
closed subset
of $\hat {M}$ and hence compact.
 Moreover, the above equivalence implies that
the foliation $\mathcal {E}$ is a regular foliation in the sense of
Palais  \cite {P}.

In order to prove that the quotient is a manifold we need to prove that
this quotient is Hausdorff. But this can be done as follows:
let $E^\perp$ be the distribution which is perpendicular, with respect
to the  metric, induced by the ambient space,
on $\hat M$. Let us define
a new Riemannian metric $\langle \, ,\, \rangle$ on
$\hat M$ by changing  the induced metric $(\, ,\, )$
on the distribution $E^\perp$ in such a way that $\rho $ is locally a Riemannian
submersion onto its image. Namely,

{\bf  . }$\langle E , E^\perp \rangle =0$.

{\bf  . } $\langle \, ,\, \rangle$ coincides with $(\, ,\, )$ when restricted
to $E$

{\bf .} $\text {d}_{\vert q}\rho$ is a linear isometry from  $(E^\perp)_q$
onto its image.

Such a metric is a bundle-like metric in the sense of Reinhart   \cite {Re}
Since $\hat M$ is compact,  $\langle \, ,\, \rangle$ is a
complete Riemannian
metric. Then, \cite {Re},Corollary 3, pp. 129, the quotient space $\bar M$ is
Hausdorff and $\pi$ is a fibration (cf. \cite {DO}, Proposition 2.4, pp. 83)

Then one has that the map $\rho $ projects down to
an immersion
$\bar {\rho}: \bar M \to \mathbb {R}^N$ and
$\rho = \bar {\rho}\circ \pi$.

\end {proof}

\

\begin {lema} \label {caustic-polar}
We keep the assumptions of Theorem \ref {3-Veronese}.
Moreover, we are in the assumptions and notation  of Case (2)(c),
inside the proof
of this theorem {\rm (in particular,   $H \simeq \text {Spin}(3)$, up
to a cover)}.

\,

\noindent \it \ \ (i) \rm All orbits of the action of $H$ on $\bar M$ have
 dimension $2$.

\,

\noindent \it \ (ii) \rm  The universal cover $\tilde {M}$ of $\bar M$
splits off a line and hence
the fundamental group of $\bar M$ in not finite (since  $\bar M$
is compact).

\end {lema}

\begin {proof}

The action of $H$ on $M_\xi$ projects down to $\bar {M}$, since
$\hat {\xi}$ is $H$-invariant and so any eigendistribution of
$\hat {A}_{\hat {\xi}}$ is $H$-invariant.
Let  $q\in M_\xi$. Then the $3$-dimensional subspace $\mathfrak {h}.q
\subset T_qM_\xi$
 intersects the $3$-dimensional horizontal subspace $\mathcal {H}_q$
in a non-trivial subspace, since
$\text {dim}(M_\xi) = 5)$.
Since we are in {\it Case (2)},
$$\{0\} \neq (\mathfrak {h}.q \cap \mathcal {H}_q)
 \subset E_1(q)$$
Let $H_{\bar q}$ be the isotropy group of  $H$ at the point
$\bar q = \bar \pi (q) \in \bar {M}$. Let $ {\mathfrak {h}}_{\bar q} =
\text {Lie}(H_{\bar q})$. Then one has that
$$ {\mathfrak {h}}_{\bar q } = \{X \in \mathfrak {h}:
X.q'  \in E_1 (q') \}$$
independent of  $q'\in S_1(q) =
(\bar {\pi})^{-1}(\bar {\pi} (\{q\}))$,
since a Killing field that is tangent to an integral manifold $S_1(q)$ of
$E_1$, at some point, must be always tangent to it
(since the action projects
down to the quotient).

If $\text {dim} ({\mathfrak {h}}_{\bar q })=3$. Then
$  {\mathfrak {h}}_{\bar q } = \mathfrak {h} $. Then
$H$ leaves invariant the $2$-dimensional integral manifold
$S_1(q)$ of $E_1$ by $q$. Then the isotropy $H_q$ has positive dimension.
But $H_q \subset H_{\pi (q)}$, where  $H_{\pi (q)}$ is the isotropy
group of $H$ at the point $\pi (q) \in M^3 = H.p$.
A contradiction, since $\text {dim}(H) = 3$.

Observe that  $\text {dim} ({\mathfrak {h}}_{\bar q })\neq  2$.
In fact, if this dimension is $2$, then
${\mathfrak {h}}_{\bar q }$ is an ideal of the $3$-dimensional
(compact type)
 Lie algebra
$\mathfrak {h}$. A contradiction, since
$\mathfrak {h}$ is simple.
We have used that a Lie subalgebra of codimension $1$ of a Lie algebra which
admits a bi-invariant metric must be an ideal. (Also, this $2$-dimensional
Lie subalgebra should
be abelian, in contradiction with $rank ({\mathfrak {h}}) =1$).

Then $\text {dim} ({\mathfrak {h}}_{\bar q })=1$
for all $q\in M_\xi$.
This implies that all $H$-orbits in $\bar M$ have dimension $2$.
 Since $H$ acts with cohomogeneity $1$ on $\bar M$
then, the universal cover of $\bar {M}$ cannot be compact. Otherwise, as it
is well-known,  there
would exist a singular orbit
(after lifting the action to the universal cover).

For the sake of self-completeness we will show the argument of this
assertion.

We  define an auxiliary Riemannian metric
 on $\bar {M}$, by changing, along the $H$-orbits, the  metric
$ \langle \, , \, \rangle $ induced by the immersion $\bar \rho$.

Since $H$ acts with cohomogeneity $1$ on $\bar M$, $H$ acts locally polarly.
In particular, the one dimensional distribution $\mathcal {D}$
on $\bar M$, perpendicular to
the $H$-orbits, is an autoparallel distribution.
If $\bar {q} \in \bar {M}$ then we put  on the orbit $H.\bar {q}$
the normal homogeneous
metric. That is, the metric associated to the reductive decomposition
$$\mathfrak {h} = \mathfrak {h}_{\bar {q}} \oplus
(\mathfrak {h}_{\bar {q}})^\perp$$
where the orthogonal complement is taken with respect to
a (fixed) bi-invariant metric on $\mathfrak {h}$.

We define $\langle \, , \, \rangle'$ by:

\noindent {\it a)}  $\langle \, , \, \rangle'_{\vert \mathcal {D}}
= \langle \, , \, \rangle_{\vert \mathcal {D}} $

 \noindent {\it b)} $\langle \mathcal  {U} , \mathcal {D} \rangle'
= 0$, where $\mathcal {U}$ is the distribution given by
the tangent spaces of the $H$-orbits on $\bar M$.

\noindent {\it c)} $\langle \, , \, \rangle'_{\vert \mathcal {U}_{\bar q}}$
coincides with the normal homogeneous metric  of
$H.\bar q$, for any $\bar {q}$ on $\bar {M}$

Since $\bar {M}$ is compact, the metric $\langle \, , \, \rangle'$
is complete. Let $\langle \, , \, \rangle'$ also denote the lift of
the Riemannian metric $\langle \, , \, \rangle'$ to the universal cover
$\tilde {M}$ of
$\bar {M} $. Then $(\tilde M , \langle \, , \, \rangle ')$ is a complete
Riemannian manifold. Let us denote by $\tilde {\mathcal {U}}$ and
$\tilde {\mathcal {D}}$ the lifts to $\tilde M$ of the distributions
$\mathcal {U}$ and $\mathcal {D}$, respectively.
Let us also lift the $H$-action on $\bar M$
to $\tilde M$
Then, since $\tilde M$
is simply connected,
the one dimension distribution $\tilde {\mathfrak {D}}$ is parallelizable.
Namely, there exists a nowhere vanishing vector field $\tilde X$ of $\tilde M$
such that $\mathbb {R}.\tilde X = \tilde {\mathfrak {D}}$.
Let $\tilde Z = \frac {1}{\Vert \tilde X \Vert} \tilde X$, where
the norm is with   the metric $\langle \, , \, \rangle '$.
Then, the flow $\phi _t $, associated to $\tilde {Z}$, is by isometries.
So  $\tilde Z$ is a Killing field.
Then $\langle \nabla _.\tilde Z, .\rangle ' $ is skew-symmetric.
So, in particular,  $\langle \nabla _v\tilde Z, v\rangle ' = 0 $,
for any  vector $v$ that lies in  $\tilde {\mathcal {U}}$.
But, if $\tilde A_{\tilde Z}$
is the shape operator of the orbit $H.x$, $x\in \tilde M$, then
$\langle \tilde A_{\tilde Z}v, v\rangle ' =
\langle \nabla _v\tilde Z, v\rangle ' = 0$. Then
$\tilde {\mathcal {U}} = \tilde {\mathcal {D}}^\perp $
is an  autoparallel distribution. The distribution
 $\tilde {\mathcal  {D}} $ is also autoparallel,
since the Killing fields induced by $H$ are always perpendicular
to it.
 But two complementary perpendicular
autoparallel distribution must be parallel.
Then, by the de Rham decomposition theorem, $\tilde M$ is a Riemannian product.
Since one of the parallel distributions is one dimensional then
$\tilde M = \mathbb {R}\times M'$.

\end {proof}

\

\begin {nota} \label {at most 4} In this paper, for dealing with
homogeneous submanifolds of dimension $3$, we need to know which are
the compact Lie groups $G$ of dimension at most $4$.
For the sake of self-completeness we will briefly
show, without using classification results,
 which are these compact Lie groups $G$ (up to covering spaces).

We will
use the following fact that it is well-known and standard
to show: \it a codimension $1$ subgroup,
of a Lie group with a bi-invariant metric,
must be a normal subgroup.\rm

\it (i) \rm $\text {dim} (G) \leq 2$.

\noindent In this case, from the above fact,
one has that $G$ must be abelian.

\it (ii) \rm $\text {dim} (G) = 3$.

\noindent If $\text {rank} (G) \geq 2$ then, from the above fact,
$G$ must be abelian.
If $\text {rank} (G) =1$, then $G$ is, up to a cover,
$\text {Spin} (3)$.
This well-known result  follows from a topological argument
that proves that a rank $1$ simply connected compact group is
isomorphic to
$\text {Spin} (3)$
(a proof can be found in Remark 2.6 of \cite {OR}).

\it (iii) \rm $\text {dim} (G) = 4$.

\noindent  If $G$ is neither  simple nor abelian,
then, from the previous
cases, we have that a finite cover of $G$
splits as $S^1\times \text {Spin}(3)$.

\noindent  If $G$ is simple then $\text {rank} (G)\leq 2$.
Otherwise $G$ would have
 a codimension $1$ (abelian) subgroup (which must be normal).

\noindent  If $G$ is simple, then $\text {rank}(G)> 1$.
Otherwise, $G= \text {Spin}(3)$
which has dimension
$3$.

\noindent  Let  $G$ be  simple and $\text {rank} (G) =  2$. Then the
${Ad}$-representation
of $G$ on $\mathfrak {g} = \text {Lie} (G)$ must have a focal (non-trivial)
orbit $G.v$. Such an orbit must have codimension $3$. The $3$-dimensional
 normal space
$\nu _v(G.v)$
  is Lie triple system, since it coincides with the commutator of $v$.
Then
 $\nu _v(G.v)$ is an ideal of
$\mathfrak {g}$.  A contradiction.

\end {nota}

\,

\begin {nota} \label {5-symmetric}
Let $X = G/K$ be an irreducible
simply connected symmetric space of the
non-compact type and rank at least $2$, where $G$ is the
connected component of the full isometry
group of $X$.  Assume that the dimension of $X$ is at most $5$.
Then, $X\simeq  \text {SL}(3)/\text {SO}(3)$.

We will next outline
a classification free proof of this fact.

 Observe,
 since $\text {rank} (X)\geq 2$, that the isotropy representation
of $K$ on $T_pX$ has a non-trivial focal orbit $M = K.v$
($p= [e]$). Such an orbit $M$ must have dimension $2$.
In fact, $M$ cannot have dimension $3$. Otherwise, a principal
$K$-orbit must have  dimension $4$ and so $K$ would act transitively
on the sphere. Observe also that the dimension of $M$ cannot be $1$.
In fact, since $K$ acts irreducibly on $T_pX$, then $K$ acts effectively on
any non trivial orbit. If $\text {dim}(M)=1$,
then $\text {dim}(K) =1$. Then, since
$\text {dim}(X) >2$,  $K$ does not act irreducibly on $T_pX$.
A contradiction.

Observe that the isotropy $K_v$ of the focal orbit $M^2 = K.v$ at
$v$ must have positive dimension (and so
$\text {dim}(K_v) = 1$). Moreover,
since $M$ is not a principal orbit,
 the image under  the slice representation
of $K_v$ is not trivial. So, by Corollary \ref {slice holonomy},
  the restricted normal holonomy group
$\Phi (v)$  of $M$ at $v$ is not trivial.
Then $\Phi (v)$ must act irreducibly on the $2$-dimensional space
$\bar {\nu}_v(M)  =
\{v\}^\perp \cap  {\nu}_v(M)$.
Observe that the codimension of $M^2$ is $3= \frac {1}{2} 2 (2 +1)$
 Then, by Proposition
\ref {n-Veronese}, $M$ is a Veronese submanifold, i.e. orthogonally equivalent
to a Veronese-type orbit $V^2$ of
$\text {SO}(3)$ on $Sim _0 (3)$ (the action is by conjugation).
So, may assume that
$Sim _0 (3) = T_pX$ and that
$M = V^2$. Then both $K$ and $\text {SO}(3)$ are Lie subgroups of
$\tilde {K} = \{ g \in \text {SO}(Sim _0 (3)): g.M = M\}$.
Observe that $\tilde {K}$ is not transitive on the unit sphere
of $Sim _0 (3)$ since the codimension of $M$ is $3$.
Let  $R'$ and $R$ be the  curvature tensors at $p=[e]$ of
$X$ and $\text {SL}(3)/\text {SO}(3)$. Then we have the following
irreducible
non-transitive holonomy systems:
$[Sim _ 0(3), R, \tilde {K}] \text { \ and \ }
[Sim _ 0(3), R', \tilde {K}]$.

Then by the holonomy theorem of Simons \ref {Simons}, $R$ is unique
up to scalar multiple and $\tilde {K} = K = \text {SO}(3)$,
since its Lie algebra is spaned by $R$.
This implies that the symmetric space $X$ is homothetical to
$\text {SL}(3)/\text {SO}(3)$.

\end {nota}

\begin {nota}\label {3-higher}

Let $M^3 = K.v \subset \mathbb {R}^N$
be a $3$-dimensional full and irreducible
 homogeneous (Euclidean) submanifold.
Assume that $\text {rank} (M)\geq 2$. In this case, by the
rank rigidity theorem,  $M$ is
an orbit of an $s$-representation. So, we may assume, that $K$-acts as
an $s$-representation.

Let $\xi$
be a $K$-invariant
parallel normal field to $M$ which is
not umbilical. If the shape operator $A_\xi$ has two different
(constant)
eigenvalues then its associated eigendistributions, let us say
$E_1$ and $E_2$ are autoparallel distributions that are invariant
under the shape operators of $M$ (recall that $A_\xi$ commutes with any
other shape operator due to Ricci equality). Then,  by the so-called Moore's
lemma \cite {BCO}, Lemma 2.7.1, $M$ is product of submanifolds. A contradiction.

If $A_\xi$ has three eigenvalues, then the multiplicities of any of them
are $1$. Since  $A_\xi$ commutes with any
other shape operator, all shape operators of $M$ must commute. Then, by the
Ricci identity, $M$ has flat normal bundle. Then $M$ is isoparametric, since it
is an orbit of an $s$-representation.

Therefore, a full irreducible and homogeneous Euclidean $3$-dimensional
submanifold $M^3$,
of higher rank, must be isoparametric with exactly
three curvature normals.
This implies that the irreducible Coxeter group associated
to $M$ \cite {Te, PT} has exactly
three reflection hyperplanes.
This is only possible if the dimension of the normal space is $2$.
Otherwise, the curvature normals must be
 mutually perpendicular and hence
$M$ would be  a product of circles.

This implies that $N=5$ and that
 $M$ is an isoparametric hypersurface of the sphere $S^4$.
Moreover, from Remark \ref {5-symmetric}, $M$ is a principal orbit
of the isotropy representation of $\text {Sl}(3)/\text {SO}(3)$.
\qed
\end {nota}

{\bf Proof of Theorem A. } If
$M^n$ is a (full) Veronese submanifold, $n\geq 3$, then
the normal holonomy, as a submanifold of the sphere,
 acts  irreducibly
 and non-transitively
(see Facts \ref {Veronese facts}, (iii)).

For the  converse observe that $M$ must be
a full and irreducible Euclidean
submanifold, since the normal holonomy group (as a submanifold
of the sphere)
 acts irreducibly (see the beginning of Section 3). Then, from
 Theorem \ref {n at least 3},
Theorem \ref   {3-Veronese}  and
Proposition \ref {n-Veronese}, $M$ is a Veronese submanifold.
\qed

\

{\bf Proof of Theorem B.}
From
Theorem \ref {3-Veronese} $M$ is an orbit of an $s$-representation.
  Assume that $\text {rank} (M) =1$.
Then, by Lemma \ref {aux}, the (restricted) normal holonomy group of $M$,
as a submanifold of the sphere,
 acts irreducibly and $N = 9 = 3 + \frac 1 2 3 (3+1)$.
Then, by Proposition \ref {n-Veronese}, $M$ is a Veronese submanifold.

If $M$ is of higher rank, then, by  Remark \ref {3-higher},
$M$ is a principal orbit
of the isotropy representation of $\text {Sl}(3)/\text {SO}(3)$.
\qed

\

\section {minimal submanifolds  with non-transitive normal holonomy}

In this section we prove Theorem C of the Introduction.

We use many of the ideas used for the homogeneous case, when
$n>3$. But now the situation is much more simple, for $n=3$.

\

{\it Proof of Theorem C.} Observe that $M$ must be full and irreducible
as a Euclidean submanifold (since the normal holonomy group,
as a submanifold of the sphere, acts irreducibly; see Section 3).
Note, by the minimality,
 that the traceless shape opertator coincides
with the shape operator (of vectors which are perpendicular to the
position vector).

We keep the notation in the proof of Theorem \ref {n at least 3}.

Let $p\in M$ be such that the adapted normal curvature tensor
$\mathcal {R}^\perp (p) \neq 0$,  or equivalently,
$R^\perp (p) \neq 0$. Let us consider the irreducible and
non-transitive holonomy systems
$[\bar {\nu}_p (M), \mathcal {R}^\perp (p), \Phi (p)] $
 and $[Sim _0 (T_pM), R , \text {SO}(T_pM)]$.

We have,
 from formula (****) of Section 3 and
 Proposition \ref {holonomy systems},
 that the shape operator at $p$, $A^p: \bar {\nu}_p(M) \to
Sim _0 (T_pM)$ is a homothecy and
$A^p\Phi (p)(A^p)^{-1} = \text {SO}(T_pM)$.
This implies, if $\phi \in \Phi (p)$,
 that the eigenvalues of $A^p_{\eta}$ coincide with
 the eigenvalues of  $A^p_{\phi (\eta)}$.

Let $U$ be a contractible neighbourhood of
$p$ in $M$ such that $\mathcal {R}^\perp $ never vanishes on $U$.

Let now $p'\in U$ be  arbitrary
and let $\gamma : [0,1] \to U$ be a
piece-wise differentiable
 curve from
$p$ to $p'$. Let $\tau _t $ be the $\nabla ^\perp$-parallel
transport along $\gamma _{[0,t]}$.

We have that $\tau _t \Phi (p)(\tau _t )^{-1}
=    \Phi (\gamma (t))   $.

Let us choose $\xi \in \bar {\nu}_p(M)$ such that $A^p_\xi
\in Sim _0 (T_pM)$ has  exactly two eigenvalues $\lambda _ 1 =
\frac {1}{2}$ of multiplicity $2$  and $\lambda _ 2= -\frac {1}{(n-2)}$
of multiplicity
$(n-2)$.

Recall that  the shape operator $A^{\gamma (t)}:
\bar {\nu }_{\gamma (t)} \to Sim _0 (T_{\gamma (t)})$ maps
$\Phi (\gamma (t))$ into $\text {SO}(T_{\gamma (t)})$.
Then, the homothecy
$g_t: = A^{\gamma (t)}\circ
\tau _t \circ (A^p)^{-1} :
Sim _0 (T_p M) \to Sim _0 (T_{\gamma (t)} M)$
maps the group
$\text {SO}(T_pM)$ into
$\text {SO}(T_{\gamma (t)}M)$. Then
$g_t$ maps the isotropy subgroup $\text {SO}(T_pM)_{A^p_\xi}
\simeq  S(\text {O}(2)\times \text {O}(n-2))$ into the isotropy
subgroup $\text {SO}(T_{\gamma (t)}M)_ {B(t)}$,
where $B(t) = A^{\gamma (t)}_{\tau _t (\xi)}$.
This implies, as it is not difficult to see,
 that $B(t)$ has two eigenvalues, let us say
$\lambda_1 ^t $ of multiplicity $2$ and
$\lambda _2 ^t$ of multiplicity
$n-2$. Since $B(t)\in Sim _0 (T_{\gamma (t)})$,
$\lambda _2 ^t = - \frac {2}{n-2}\lambda _1 ^t$.

Then the two eigenvalues of ${B(t)}$ are
constant up to the multiplication by  $a(t) = \lambda _1^t \neq 0$.
 Note, if $\gamma $ is a loop by $p$, that $\tau _1 \in
\Phi (p)$. Then, as we have previously observed, the eigenvalues
of $A^p_{\xi} $ are the same as those of ${B(1)}$.
Then $a(t)$ depends only on $\gamma (t)$. So there is
a non-vanishing
$f: U\to \mathbb {R}$ such that $a(t) = f(\gamma (t))$. It is
standard to show that $f$ must be $C^\infty$.
Note that $f(p) =
\frac {1}{2}$.

Let us consider (eventually,  by making $U$  smaller) the holonomy
tube $U_{\xi}$. We use the notation in the proof to Theorem \ref
{n at least 3}. We will modify  the arguments in this proof.

We have the parallel normal field $\hat {\xi}$ of $U_{\xi}$.
The eigenvalues of the   shape operator
$\hat {A}_{\hat {\xi}}$   at $q\in U_{\xi}$
are given by

$$\hat {\lambda} _1 (q) =
\frac {f(\pi (q))}{1 - f(\pi (q))} $$
associated to the (horizontal) eigendistribution $E_1$
of dimension $2$

$$\hat {\lambda} _2 (q) =
\frac {-\frac {f(\pi (q))}{n-2}}
{1 + \frac {f(\pi (q))}{n-2}}$$
associated to the (horizontal)
eigendistribution $E_2$
of dimension $n-2$.

The third eigenvalue of $\hat {A}_{\hat {\xi}}$, is
$\hat {\lambda}_3 = -1$, associated to the vertical
distribution $\nu$, tangent to the normal holonomy orbits.

By the Dupin condition, $\text {d}(\hat {\lambda}_1)( E_1) = 0 $
which implies that
  $$\text {d}(f\circ \pi) (E_1 ) = 0 \ \ \ \ \ \text {(J)}$$
 If $n>3$ this is
also true for the eigendistribution $E_2$, since it has dimension at
least $2$. But we will not assume this and the proof will also work
for $n=3$.

From the tube formula, as we have observed in the proof of
Theorem \ref {3-Veronese}, Case (2), (b), $\text {d} \pi (E_1 (q)) =
E_1 (q)$, as linear subspaces. Moreover,
$E_1(q)$ is an eigenspace of $A_{q-\pi (q)} = A_{\hat {\xi}(q)}$,
 where $A$ is the shape operator of
$M$ (we drop  the supra-index $\pi (q)$ of $A$). Let now
$q \in U_\xi$ with $\pi (q) = p$ and let $\mathbb {V}$ be the
subspace of $T_pM$ which is generated by
$E_1(q')$, with
$q' \in
\Phi (p).q = (\pi ^{-1}(\{p\}))_q$. If $\mathbb {V} = T_pM$,  then, from
formula (J), $\text {d}f(T_pM) = \{0\}$. If $\mathbb {V}$ is properly
contained in $T_pM$, then let $0\neq v \in \mathbb {V}^\perp$. We will derive,
in this case, a contradiction. In fact, since any shape operator
$A_{q'-p}$ has only two eigenvalues and $v$ is perpendicular to the eigenspace
$E_1 (q')$ of $A_{q'-p}$, then $v$ is an eigenvector of this shape operator, for any
$q'\in \Phi (p).q$. Observe that the linear span of
$\Phi (p).q$ is $\bar {\nu}_p (M)$, since $q'\neq 0$ and
$\Phi (p)$ acts irreducibly
on this normal space. Then $v$ is a common eigenvector for all shape operators
$A_\eta$, $\eta \in \nu _p (M)$. But $A: \bar {\nu}_p(M) \to
Sim _0(T_pM)$ is an isomorphism. This is a contradiction.
Then $\text {d}f(T_pM) = \{0\}$ and the same is valid for all
$p'\in U$. Then $f = f (p) = \frac {1}{2}$ is constant on $U$.

 Then the eigenvalues $\hat {\lambda }_1, \hat {\lambda }_2 ,
\hat {\lambda }_3$ are constant on $U_{\xi}$.  Then $\hat {\xi}$ is a
(non-umbilical) parallel
normal isoparametric field  of $U_{\xi}$. Then, by \cite {CO} (see
\cite {BCO}, Theorem 5.5.2),  $U_\xi$ and hence $U$ has
constant principal curvatures. But this is true provided one shows that
$U_{\xi}$ is full and locally irreducible around some point
$q\in \pi^
{-1}(\{p\})$.

\noindent Let us show  that the local normal holonomy group of
$M$ at $p$ coincides with the restricted normal holonomy group.
In fact, the holonomy system
$[\bar {\nu}_p (M), \mathcal {R}^\perp (p), \Phi (p)] $
is irreducible and non-transitive. Then,
by the holonomy  theorem of Simons \cite {S}, it is symmetric. Moreover,
$\text {Lie}(\Phi (p))$ is linearly generated by the
endomorphisms $\{\mathcal {R}_{\xi , \eta} ^\perp (p)\}$.
This implies that the local normal holonomy at $p$ coincides with
$\Phi (p)$. Then the local rank of $M$, as submanifold of the Euclidean
space, is $1$.
This implies that $M$ is full and locally irreducible around
$p$. Hence  $U_{\xi}$ is full and irreducible around
any point
$q\in \pi ^
{-1}(\{p\})$.
Then $U$ is a submanifold with constant principal curvatures.

Since the normal holonomy of $M$ is not transitive on the unit sphere,
of the normal space to the sphere, any principal holonomy tube
(which is isoparametric)
has codimension at least $3$ in the Euclidean space.
 Then, by the theorem of
Thorbergsson \cite {Th}, $U$ is locally an orbit of an $s$-representation.
Then $\Vert \mathcal {R}^\perp\Vert$ is constant on $U$. From this one
obtains that $\Vert \mathcal {R}^\perp\Vert$ is constant on $\Omega$, where
$\Omega$ is a
connected component
of the open subset $ \{p\in M: \mathcal {R}^\perp(p)  \neq 0\}$.
  But if $p'\in M$ is a limit point of $\Omega$ then,
$\mathcal {R}^\perp(p')  \neq 0$. This implies  that $\Omega $ can be enlarged
unless $p'\in \Omega$. This shows that the open subset $\Omega $ is also
closed in $M$.
Then $M$ has constant principal curvatures. Hence,
the image of $M$ (under the isometric immersion), is an embedded
submanifold with constant principal curvatures. Moreover, it is
 an orbit of an $s$-representation. From Proposition \ref {n-Veronese},
the image of $M$ is a Veronese submanifold.

The converse is true by Facts \ref {Veronese facts}, (i) and (iii).
\qed

\begin {nota}\label {itoh ogiue}

We keep the notation of the proof of Theorem C.
The fact that $f$ is constant can also be proved
in the following way. Let $p\in M$ be such that
$\mathcal {R}^\perp (p)\neq 0$. Then, since
the shape operator $A$ maps $\Phi (p)$-orbits
into $\text {SO}(T_pM)$-orbits of
$Sim _0 (T_pM)$, one obtains that
the second fundamental form is $\lambda$-isotropic.
That is, the length  of $\alpha (X,X)$ is $\lambda (p)$
independent
of $X$ in the unit sphere of $T_pM$, where $\alpha$ is
the second fundamental form. The function $\lambda$ must
be a constant multiple of $f$. Then, by Proposition
4.1. of \cite {IO}, $\lambda$, and hence $f$, must be
constant ($n\geq 3$).
\end {nota}

\,

\section {The number of
factors of the normal holonomy}

In this section we will prove a sharp linear bound, depending
on the dimension $n$ of the submanifold, of the number
of irreducible factors of the local normal holonomy
representations. This improves, substantially, the
quadratic bound
 $\frac {1}{2}n(n-1)$ given in
Theorem 4.5.1 of \cite {BCO}.

\begin {prop}\label {cota} Let $M^n$ be a submanifold of the
Euclidean space $\mathbb {R}^N$. Assume that at any point of $M$ the local
normal holonomy group and the restricted normal holonomy group coincide
(or, equivalently, the dimensions
of the local normal holonomy groups are constant on $M$).
 Let  $p\in M$ and let
 $r$ be the number of irreducible (non-abelian)
subspaces  of the representation of the restricted  normal holonomy
group $\Phi (p)$ on $\nu _p (M)$.
 Then $r\leq \frac {n}{2}$.
 Moreover, this bound is sharp for all $n\in \mathbb {N}$
(also in the class of irreducible  submanifolds).
\end {prop}

\begin {proof}

Let us decompose
$\nu _p (M) = \nu _p^0(M)  \oplus \nu _p^ 1(M) \oplus ...
\oplus \nu _p^r (M)$, where $\Phi (p)$ acts trivially
on $\nu _p^0 (M)$ and irreducibly on $\nu _p^i (M)$,
for $i=1, ... , r$. From the assumptions we obtain that
  $\nu _p^i (M)$ extends to
a $\nabla ^\perp$-parallel subbundle $\nu ^i$
of the normal bundle
$\nu (M)$, $i=0, ..., r$ (eventually, by making $M$ smaller around $p$).
Note that we have the decomposition
$\nu (M) = \nu ^0(M)  \oplus \nu ^ 1 (M) \oplus ...
\oplus \nu ^r (M)$. Moreover, we obtain from  the
assumptions, for any $q\in M$, that  the
local normal holonomy group $\Phi (q)$ acts
 trivially  on $\nu _q ^0(M)$  and irreducibly on
 $\nu _q ^i (M)$, for any $i= 1, ..., r$.

Let $\mathcal {R}^\perp_{\xi  , \xi ' }$ be the adapted
normal curvature tensor (see Section 1.1).
From the expression of $\mathcal {R}^\perp$ in terms of
shape operators $A$, one has that
$\mathcal {R}^\perp_{\xi  , \xi ' } =0 $
if and only if $[A_{\xi }, A_{\xi '}] = 0$.

 Observe, if $i\neq j$,  that
$ \mathcal {R}^\perp_{\xi _i  , \xi ' _j } = 0$ if
$\xi _i , \xi ' _j$ are  normal sections that lie in $\nu ^i(M)$ and
 $\nu ^j(M)$, respectively.

There must exist $q\in M$, arbitrary close to $p$, such that
$ \mathcal {R}^\perp_{ \nu _q ^i , \nu _q ^i } \neq \{ 0\}$,
for all $i=1, ..., r$. In fact, there exists $q_1 \in M$,
arbitrary close to $p$ such that
$ \mathcal {R}^\perp_{ \nu _{q_1} ^1 , \nu _{q_1} ^1 } \neq \{ 0\}$
(otherwise, $\nu ^1 (M)$ would be flat). The above inequality must be
true in a neighbourhood $V_1$ of $q_1$. Now choose $q_2 \in V_1$ such
that $ \mathcal {R}^\perp_{ \nu _{q_2} ^2 , \nu _{q_2} ^2 } \neq \{ 0\}$.
Continuing with this procedure we find $q : = q_r$ with
the desired properties.

Let us show that  for any
$i = 1, .... , r$ there exist $\xi _i , \xi '_i$  en $\nu _q ^i(M)$
such that
$[A_{\xi _i} , A_{\xi '_i}]$ does not belong to the algebra
of endomorphisms
generated by
$\{A_{\eta ^i}\}$, where $\eta ^i \in \nu _q (M)$ has no component in
$\nu _q^i (M)$. In fact, if this is not true, then,
 for any  $\xi _i , \xi '_i$  in $\nu _q ^i(M)$,
  $[A_{\xi _i} , A_{\xi '_i}]$ commutes with $A_{\xi _i}$
(since the shape operators of elements of the subspaces $\nu _q^j(M)$
commute with  con
$A_{\xi _i}$, if $j\neq i$). Then
$$\langle [ [A_{\xi _i} , A_{\xi '_i}], A_{\xi _i}] , A_{\xi '_i}\rangle = 0
  =-\langle  [A_{\xi _i} , A_{\xi '_i}], [A_{\xi _i} , A_{\xi '_i}] \rangle$$
and hence   $ [A_{\xi _i} , A_{\xi '_i}] = 0$.
A contradiction, since
$ \mathcal {R}^\perp_{ \nu _q ^i , \nu _q ^i } \neq \{ 0\}$.
This proves our assertion.

Observe that
$[A_{\xi _1} , A_{\xi '_1}], ... ,  [A_{\xi _r} , A_{\xi '_r}] $
are linearly independent and commuting
skew-symmetric endomorphisms of $T_qM$.
Then $r\leq rank (\text {SO}(T_pM)) = [\frac {n}{2}]$
(the integer part of $\frac n 2$).
This proves the inequality.

Let us see that it is sharp.
For $M^2 \subset S^{k_1 -1}  , \bar {M}^3 \subset S^{k_2-1}$
 be a surface and a  $3$-dimensional submanifold and such that
the  normal holonomies have one irreducible factor (for example, the Veronese
$V^2$ and $V^3$).
Let $n >3$ and write $n = 2d $ is $n$ is even  or $n = 2d +3 $ if $n$ is odd.

Let $M^n$ be the product of $d$ times $M^2$ or  $M^n$ be the product of $d$ times
$M^2$ by $\bar {M}^3$. Such submanifolds are contained in the product of
Euclidean ambient spaces.
Moreover, the number of irreducible factors of the normal holonomy
group (representation) of $M^n$
is exactly the upper bound $[\frac {n}{2}]$.
Moreover, since $M^n$ is contained in a sphere, we can apply
to $M^n$ a conformal transformation of the sphere
(the normal holonomy group is a conformal invariant) in such a way that
$M^n$ is an irreducible (Riemannian) submanifold of the Euclidean space.

 \end {proof}

\,

\section  {Further comments}

\begin {nota} \label {LP}  There is a beautiful  result of
Little and Phol  \cite {LP} which
characterizes Veronese
submanifolds $M^n$, modulo projective diffeomorphisms,
 by the {\it two-piece property}  and
the fact that the codimension is the maximal one
$\frac 1 2n(n+1)$ (for submanifolds with the two-piece
property). Note that a tight submanifold has the two-piece
property. This result generalizes the well-know result of
Kuiper for $n=2$. A projective transformation, in general,
does not preserve the normal holonomy (unless it induces a conformal
transformation of the ambient sphere).

\end {nota}

\,

\begin {nota}
\label {conformal}  A natural question that arises, since the normal
holonomy group is a conformal invariant, is the following:
{\it is a compact submanifold $M^ n \subset S^ {n-1 + \frac 1 2 n(n+1)}$,
with irreducible and non-transitive (restricted) normal holonomy,  equivalent,
modulo conformal transformations of the
sphere, to a Veronese  submanifold?.}
\end {nota}

\,

\begin {nota}\label {examples} The symmetric space
$X = \text {SU}(4)/\text {SO(4)}$,
dual to $\text {Sl}(4)/\text {SO(4)}$, is isometric to
the Grassmannian  $\text {SO(6)}/\text {SO(3)}\times \text {SO(3)}$.
In this last model, $T_{[e]}X = \mathbb {R}^{3\times 3}$ and the
 isotropy representation is given by  $(g,h).T = gT h^{-1}$,
$(g,h) \in \text {SO(3)}\times \text {SO(3)}$. The Veronese
submanifold $V^3$ is given by
$$\text {SO(3)}\times \text {SO(3)}.Id =
\text {SO(3)}\times \{Id\}.Id = \text {SO}(3) \subset
\mathbb {R}^{3\times 3}$$
Thus  $V^3$ is also
an orbit of the
smaller group $\text {SO}(3) \simeq \text {SO(3)}\times \{Id\}$.
The other orbits  $\text {SO}(3).A$, where $A$ is invertible and
near $Id$, must be
full and irreducible submanifolds of
$\mathbb {R}^{3\times 3}$, since $V^3$ is so.
Note that the action of $\text {SO}(3)$ on
$\mathbb {R}^{3\times 3}$ is reducible. In fact, it is
the sum of three times the standard representation
of $\text {SO}(3)$ on $\mathbb {R}^3$.
The orbit, $\text {SO}(3).A$ is not minimal in the sphere, for
$A$ generic. So, the normal holonomy holonomy group of this orbit
must be transitive on the unit sphere (of the normal space to
the sphere).

Observe that the linear isomorphism $r_{A^{-1}}$ of
$\mathbb {R}^{3\times 3}$,
$r_{A^{-1}}(T) =  TA^{-1}$, transforms $\text {SO}(3).A$
into $V^3$. In particular, since $V^3$ is a
tight submanifold,
that orbit is so. Hence, as it is well known,
  $\text {SO}(3).A$\,  is a taut submanifold,
since it lies in a sphere (see \cite {CR, G}).

\end {nota}

\newpage

\section  {Appendix}

\subsection {The Veronese embedding}

\

We recall here some  basic definitions and
facts about the well-known
 Veronese submanifolds.

Let $S^n$, $n\geq 2$, be the unit sphere of the Euclidean space
$\mathbb {R}^{n+1}$ and let $\mathbb {R}^{n+1}
 {\otimes}_s
\mathbb {R}^{n+1}$ be space of  symmetric $2$-tensors of
$\mathbb {R}^{n+1}$. Let $h: \mathbb {R}^{n+1}
 {\otimes}_s
\mathbb {R}^{n+1} \to Sim (n +1)$ the usual isomorphism
onto the symmetric matrices of $\mathbb {R}^{n+1}$.
Namely, let  $e_1, ... , e_{n+1}$ be the canonical basis
of $\mathbb {R}^{n+1}$.  Then,
$h (e_i\otimes e_j +  e_j\otimes e_i) $ is the matrix
whose  coefficients $a_{k,l}$ are all zero except:
$$a_{i,j} = a _{j,i} = 1, \text { if } i\neq j; \ \ \
a_{i,i} = 2, \text { if } i= j$$
\indent The Veronese map $Q: S^n \to Sim (n + 1)$ is defined by
$$\text {Q}(v) = h(v\otimes v)$$
 Observe that $(\text {Q}(v))_{i,j} =
v_iv_j$, where $v = (v_1, ... , v_{n+1})$.
Let $\langle \, , \,\rangle$ be the inner
product on $Sim (n +1)$ given by $\langle A , B\rangle
= \frac 1 2  {trace}(AB)$.
Then $\text {Q}$ is an isometric immersion. Observe that
$ {trace}(\text {Q}(v)) = 1$, for all $v\in S^n$. So, the image
of $Q$ is contained in the affine hyperplane of $Sim (n +1)$, given by
the linear equation
$$\langle \, \cdot \,  , Id\rangle = \frac 1 2$$

 Let $\tilde {\rho} : S^n \to Sim _0(n + 1)$ be defined by
$\tilde {\rho} (v) = Q (v) - \frac {1}{n+1}Id$, where
 $Sim _0(n + 1)$ are  the symmetric traceless matrices.
The map $\tilde {\rho}$ is called  the {\it Veronese
Riemannian immersion} of the
sphere $S^n$ into
$Sim _0(n + 1)$. One has that $\tilde {\rho}$,
 (as well as $\text {Q}$)
is $\text {O}(n+1)$-equivariant.
Namely, if $g\in \text {O}(n+1)$, then
$$\tilde {\rho} (g.v) = g.\tilde {\rho}(v).g^{-1}$$
 In fact, if we regard
$v\in \mathbb{R}^{n+1}$ as a column vector, then
$$\tilde {\rho} (v) =
v.v^t  - \frac {1}{n+1}Id$$

From the above formula it  follows easily the $\text {O}(n +1)$-equivariance of
$\tilde {\rho}$. It is also not difficult to verify, as it is well known,
that $\tilde {\rho}(v) = \tilde {\rho}(w)$ if and only if $w=\pm v$. Therefore,
$\tilde {\rho}$ projects down to an isometric $\text {O}(n+1)$-equivariant
embedding  $\rho : \mathbb {R}P^n \to Sim _0(n+1)$, the so-called
{\it Veronese Riemannian embedding}.
\

Let us consider the  simple symmetric  pair
$(\text {Sl}(n+1), \text {SO}(n+1))$ of the non-compact type.
The Cartan decomposition associated to such a pair is
$$\mathfrak {sl}(n+1) = \mathfrak {so}(n+1) \oplus Sim _0(n+1)$$
Then
the (irreducible) isotropy representation of
$X =  \text {Sl}(n+1)/ \text {SO}(n+1)$ is naturally identified with the
action, by conjugation,  of $\text {SO}(n+1)$ on $Sim _0(n+1)$.
Then, the image of the Veronese embedding, is the orbit
$$M = \text {SO}(n+1).S$$
where $S\in Sim_0(n+1)$ is the diagonal matrix with exactly two eigenvalues. Namely,
$1-\frac {1}{n+1}$ and $-\frac {1}{n+1}$. The first one,
with multiplicity $1$,  is associated to the
eigenspace $\mathbb {R}e_1$ and the second one,
with multiplicity $n$, is associated to the eigenspace
$(\mathbb {R}e_1)^\perp$.

Let $S'\in Sim_0(n+1)$ with exactly two eigenvalues $\lambda _1$ of multiplicity
$1$ and $\lambda _2$ with multiplicity $n$. Assume that
$\Vert S'\Vert =  \Vert S \Vert $ (i.e. $S$ and $S'$ have the same length).
It is easy to verify that either $\lambda _1 = 1-\frac {1}{n+1}, \lambda _2
= -\frac {1}{n+1}$ or $\lambda _1 = -1+\frac {1}{n+1}, \lambda _2
= \frac {1}{n+1}$. In the first case one has that
$S'\in S0(n+1).S = \rho (\mathbb {R}P^n)$. In the second case,
$-S'\in S0(n+1).S$.

Observe that $S'$ and $-S'$ cannot be both in the
image of the Veronese embedding, since the
respective eigenvalues of multiplicity
$1$ are different.
In general, if $\bar {S}\in Sim _0 (n+1)$ has
two different eigenvalues, one of multiplicity $1$ and the other of multiplicity $n$,
then $\bar {S} = \lambda S$, for some $0\neq \lambda \in \mathbb {R}$.
The orbit $\text {SO}(n+1).\bar {S}$ is called
a {\it Veronese-type orbit} (see Section 1.1).
Observe that there are exactly two Veronese-type
orbits in  any given sphere, centered at $0$, of $Sim _0(n+1)$. Moreover,
any of these two Veronese-type
orbits is isometric to the other,
 via the isometry $-Id_{Sim _0(n+1)}$ of $Sim _0 (n+1)$.

We have the following well-known fact.

\begin {lema}\label {minima}

Let $\text {SO}(r )$ acts by conjugation on  $Sim _0(r)$, the
 traceless symmetric
 $r\times r$-matrices,  and let $M = \text {SO}(r).A$ be an orbit,
$A\neq 0$.
Then $r -1\leq \text {dim}(M)$. Moreover, the equality holds
 if and only if $M$ is an
orbit of Veronese-type.

\end {lema}

\begin {proof} Let us assume that $M$ has minimal dimension.
We will  first prove that $A$ has exactly two eigenvalues.
If not, let $\lambda _1 , ... , \lambda _d$ be the different eigenvalues
of $A$ with associated eigenspaces $E_1, ... , E_d$
($d\geq 3$).
Then the isotropy subgroup $\text {SO}(r)_A =
\text {S}(\text {SO}(E_1)\times ... \times \text {SO}({E_d}))$
has less dimension than
 $\text {S}(\text {SO}(E_1) \times \text {SO} ({E_2 \oplus ... \oplus E_d}))$,
which is the isotropy group of some $\neq A'\in Sim _0 (r)$ with
two different eigenvalues
whose associated eigenspaces are $E_1$ and $E_2 \oplus ... \oplus E_d$. Then
$\text {dim} (M) > \text {dim} (\text {SO}(r).A')$.
A contradiction. Therefore, $d=2$. (Observe that $d=1$ implies that $A = 0$, since
it is traceless).

Let now $k=\text {dim} (E_1)$ and so $r-k=\text {dim} (E_2)$.

We have  the well known formula  for the dimension of the Grassmannians,
 $$\text {dim} (M) =
\text {dim}(\text {SO}(r)) - \text {dim}(\text {SO}(k)) -
\text {dim}(\text {SO}(r -k))= k(r-k)$$
But the quadratic $q(x)= x(r-x)$,
$x\in [0,r]$,  is increasing in the
interval
$[0, r/2)$ and it is
decreasing  in $(r/2, r]$. So, the minimum  of $q$,
restricted
to the finite set $\{1, ... , r-1\}$ is attained at both, $x=1$ and $x=r-1$.
Then $k=1$ or $k=r-1$, in which case $M$ is a Veronese-type orbit (of dimension
$r-1$).

\end {proof}

\end {document}